\documentclass[12pt]{article}

\usepackage{amsthm,amsmath,amssymb}
\usepackage[round]{natbib}
\usepackage{mathrsfs}
\usepackage{amsfonts}
\usepackage{enumerate}
\usepackage{latexsym}
\usepackage{multirow}
\usepackage{cite}
\usepackage[pdftex]{color,graphicx}
\usepackage[colorlinks,citecolor=blue,urlcolor=blue]{hyperref}

\newcommand{\argmin}{\operatornamewithlimits{\textrm{argmin}}}

\newtheorem{thm}{Theorem}[section]

\newtheorem{lemma}[thm]{Lemma}

\newtheorem{remark}{Remark}[section]

\newcommand{\E}{{\mathbb E}}

\newcommand{\R}{{\mathbb R}}
\renewcommand{\P}{{\mathbb P}}

\newcommand{\C}{{\mathcal{C}}}
\newcommand{\mix}{{\mathfrak{R}}}

\newcommand{\ic}{{\mathfrak{I}}}
\newcommand{\lin}{{\mathfrak{L}}}

\newcommand{\cur}{\mathfrak{K}}

\newcommand{\Ps}{{\mathcal{P}}}

\newcommand{\F}{{\cal F}}

\newcommand{\X}{{\mathcal{X}}}

\newcommand{\ham}{{\Upsilon}}

\newcounter{rcnt}[section]

\def\qt#1{\qquad\text{#1}}

\def\argmin{\mathop{\rm argmin}}

\begin{document}

\title{Global risk bounds and adaptation in univariate convex regression}

\author{Adityanand Guntuboyina\footnote{Supported by NSF Grant DMS-1309356}  $\;\;\;$ and $\;\;\;$ Bodhisattva Sen\footnote{Supported by NSF Grants DMS-1150435 and AST-1107373} \\ University of California at Berkeley and Columbia University}
\date{}
\maketitle

%
%


\begin{abstract}
  We consider the problem of nonparametric estimation of a convex
regression function $\phi_0$. We study the risk of the least squares
estimator (LSE) under the natural squared error loss. We show that the
risk is always bounded from above by $n^{-4/5}$ modulo logarithmic
factors while being much smaller when $\phi_0$ is well-approximable
by a piecewise affine convex function with not too many affine pieces (in
which case, the risk is at most $1/n$ up to logarithmic factors). On the
other hand, when $\phi_0$ has curvature, we show that no estimator can
have risk smaller than a constant multiple of $n^{-4/5}$ in a very
strong sense by proving a ``local'' minimax lower bound. We also study
the case of model misspecification where we show that the LSE exhibits
the same global behavior provided the loss is measured from the
closest convex projection of the true regression function. In the
process of deriving our risk bounds, we prove new results for the
metric entropy of local neighborhoods of the space of univariate
convex functions. These results, which may be of independent interest,
demonstrate the non-uniform nature of the space of univariate convex
functions in sharp contrast to classical function spaces based on
smoothness constraints. 
\end{abstract}





\section{Introduction}
We consider the problem of estimating an unknown convex function
$\phi_0$ on $[0, 1]$ from observations $(x_1, Y_1), \dots, (x_n, Y_n)$
drawn according to the model
\begin{equation}\label{eq:Model}
  Y_i = \phi_0(x_i) + \xi_i, \qt{for $i = 1, \dots, n,$}
\end{equation}
where $x_1, \dots, x_n$ are fixed points in $[0, 1]$ and $\xi_1,
\dots, \xi_n$ represent independent mean zero errors. Convex regression is an important problem in the general area of nonparametric estimation under shape constraints. It often arises in applications: typical examples appear in economics (indirect utility, production or cost functions), medicine (dose response experiments) and biology (growth curves). 

The most natural and commonly used estimator for $\phi_0$ is the full
least squares estimator (LSE), $\hat{\phi}_{ls}$, which is defined as any minimizer of the LS criterion, i.e.,   
\begin{equation*}
  \hat{\phi}_{ls} \in \argmin_{\psi \in \C} \sum_{i=1}^n \left(Y_i -
    \psi(x_i) \right)^2,
\end{equation*}
where $\C$ denotes the set of all real-valued convex functions on $[0, 1]$. $\hat{\phi}_{ls}$ is not unique even though its values at the data points $x_1, \dots, x_n$ are unique. This follows from that fact that $(\hat{\phi}_{ls}(x_1),\ldots, \hat{\phi}_{ls}(x_n)) \in \R^n$ is the projection of $(Y_1,\ldots, Y_n)$ on a closed convex cone. A simple linear interpolation of these values leads to a unique continuous and piecewise linear convex function with possible knots at the data points, which can be treated as the canonical LSE. The canonical LSE can be easily computed by solving a quadratic program with $(n-2)$ linear constraints.  

Unlike other methods for function estimation such as those based on kernels which depend on tuning parameters such as smoothing bandwidths, the LSE has the obvious advantage of being completely automated. It was first proposed by \citet{Hildreth54} for the estimation of production functions and Engel curves. Algorithms for  its computation can be found in \citet{Dykstra83} and \citet{FraserM89}.  The theoretical behavior of the LSE has been investigated by many authors. Its consistency in the supremum norm on compact sets in the interior of the support of the covariate was proved by \citet{HanPled76}. \citet{Mammen91} derived the rate of convergence of the LSE and its derivative at a fixed point, while \citet{GroeneboomEtAl01}  proved consistency and derived its asymptotic distribution at a fixed point  of positive curvature. \citet{DuembgenEtAl04} showed that the supremum distance between the LSE and $\phi_0$, assuming twice differentiability, on a compact interval in the interior of the support of the design points is of the order $(\log(n)/n)^{2/5}$.      

In spite of all the above mentioned work, surprisingly, not much is known about the \textit{global} risk behavior of the LSE under the natural loss function: 
\begin{equation}\label{eq:LossFn}
  \ell^2(\phi, \psi) := \frac{1}{n} \sum_{i=1}^n \left(\phi(x_i) -
    \psi(x_i) \right)^2. 
\end{equation} 
This is the main focus of our paper. In particular, we satisfactorily address the following questions in the paper: At what rate does the risk of the LSE $\hat{\phi}_{ls}$ decrease to zero? How does this rate of convergence depend on the underlying true function $\phi_0 \in \C$; i.e., does the LSE exhibit faster rates of convergence for certain functions $\phi_0$? How does $\hat{\phi}_{ls}$ behave, in terms of its risk, when the model is misspecified, i.e., the regression function is not convex?


We assume, throughout the paper, that, in \eqref{eq:Model}, $x_1 < x_2 < \dots < x_n$ are fixed design points in $[0, 1]$ satisfying 
\begin{equation}\label{eq:DesignPts}
c_1 \leq n(x_i - x_{i-1})
\leq c_2, \quad \mbox{ for } i =2,3,\ldots,  n,
\end{equation} 
where $c_1$ and $c_2$ are positive constants, and that $\xi_1, \ldots,
\xi_n$ are independent normally distributed random variables with mean
zero and variance $\sigma^2 > 0$. In fact, all the results in our paper, excluding those in Section~\ref{LowerBd}, hold under the milder assumption of subgaussianity of the errors. Our contributions in this paper can be summarized in the following.

\begin{enumerate}
	\item We establish, for the first time, a finite sample upper
          bound for risk of the LSE $\hat{\phi}_{ls}$ under the loss
          $\ell^2$ in Section~\ref{RiskAna}. The analysis of the risk behavior of
          $\hat{\phi}_{ls}$ is complicated due to two facts:
          (1) $\hat{\phi}_{ls}$ does not have a closed form
          expression, and (2) the class $\C$ (over which
          $\hat{\phi}_{ls}$ minimizes the LS criterion) is not totally
          bounded. Our risk upper bound involves a minimum of two  
          terms; see Theorem~\ref{twinf}. The first term 
          says that the risk $\E_{\phi_0} \ell^2(\hat{\phi}_{ls},
          \phi_0)$ is bounded by $n^{-4/5}$ up to logarithmic
          multiplicative factors in $n$. The second
          term in the risk bound says that the risk is bounded from
          above by a combination of the parametric rate $1/n$ and an
          approximation term that dictates how well $\phi_0$ is
          approximated by a piecewise affine convex
          function (up to logarithmic multiplicative
          factors). Our risk bound, in addition to  
          establishing the $n^{-4/5}$ worst case bound, implies that
          $\hat{\phi}_{ls}$ adapts to piecewise affine convex
          functions with not too many pieces (see Section~\ref{RiskAna} for the precise definition). This is remarkable
          because the LSE minimizes the LS criterion over
          all convex functions with no explicit special treatment for
          piecewise affine convex functions.

%
  
  \item In the process of proving our risk bound for the LSE, we prove
    new results for the metric entropy of balls in the space of convex
    functions. One of the
    standard approaches to finding risk bounds 
    for procedures based on empirical risk minimization (ERM) says
that the risk behavior of $\hat{\phi}_{ls}$ is determined by the
metric 
entropy of balls in the parameter space around the true function (see,
    for example,~\citet{VandegeerBook, BM93, vaartwellner96book,
      Massart03Flour}). The 
ball around $\phi_0$ in $\C$ of radius $r$ is defined as
\begin{equation}\label{lb}
S(\phi_0,
r) := \{\phi \in \C: \ell^2(\phi, \phi_0) \leq r^2 \}.  
\end{equation}
Recall that, for a subset $\F$ of a metric space $(\X,\rho)$, the
$\epsilon$-covering number of $\F$ under the metric $\rho$ is denoted
by $M(\epsilon, \F, \rho)$ and is defined as the smallest number of
closed balls of radius $\epsilon$ whose union contains $\F$. Metric
entropy is the logarithm of the covering number.  

We prove new upper bounds for the metric entropy of $S(\phi_0,
r)$ in Section~\ref{MetricEnt}. These bounds depend crucially on
$\phi_0$. When $\phi_0$ is a  piecewise affine function with not too many pieces, the metric entropy of $S(\phi_0, r)$ is much smaller than when $\phi_0$ has a second derivative that is bounded from above and below by positive constants. This difference in the sizes of the balls $S(\phi_0,
r)$ is the reason why $\hat{\phi}_{ls}$ exhibits different rates for different convex functions $\phi_0$. It should be noted that the convex functions
$S(\phi_0, r)$ are not uniformly bounded and hence existing results on
the metric entropy of classes of convex functions
(see~\citet{Bronshtein76, Dryanov, GS13}) cannot be used directly to
bound the metric entropy of $S(\phi_0, r)$. Our main risk bound Theorem~\ref{twinf} is proved in Section~\ref{prbs} using the developed metric entropy bounds for $S(\phi_0, r)$. These new bounds are also of independent interest.

	\item We investigate the optimality of the rate $n^{-4/5}$. We show
  that for convex functions $\phi_0$ having a bounded (from both above
  and below) curvature on a sub-interval of $[0, 1]$, the rate
  $n^{-4/5}$ cannot be improved (in a very strong sense) by any other
  estimator. Specifically we show that a certain ``local'' minimax
  risk (see Section~\ref{LowerBd} for the details), under the loss
  $\ell^2$, is bounded from below by $n^{-4/5}$. This shows, in
  particular, that the same holds for the global minimax rate for this
  problem. 
  

\item We also provide risk bounds in the case of model
  misspecification where we do not assume that the underlying
  regression function in~\eqref{eq:Model} is convex. In this case we
  prove the exact same upper bounds for $\E_{\phi_0}
  \ell^2(\hat{\phi}_{ls}, \phi_0)$ where $\phi_0$ now denotes any
  convex projection (defined in Section~\ref{Misspec}) of the unknown
  true regression function. To the best of our knowledge, this is the
  first result on global risk bounds for the estimation of convex
   regression functions under model misspecification.  Some auxiliary results about convex functions useful in the proofs of the main results are deferred to Section~\ref{AuxRes}. 
\end{enumerate}
Two special features of our analysis are that: (1) all our risk-bounds
are non-asymptotic, and (2) none of our results uses any (explicit)
characterization of the LSE (except that it minimizes the least
squares criterion) as a result of which our approach can, in
principle, be extended to more complex ERM procedures, including shape
restricted function estimation in higher dimensions; see e.g.,
\citet{SS11}, \citet{SW10} and \citet{CSS10}.

Our adaptation behavior of the LSE implies in particular that the LSE
converges at different rates depending on the true convex function
$\phi_0$. We believe that such adaptation is rather unique to problems
of shape restricted function estimation and is currently not very well
understood. For example, in the related problem of monotone function 
estimation, which has an enormous literature (see e.g., \citet{G56},
\citet{Birge89}, \citet{Zhang02} and the references therein), the only
result on adaptive global behavior of the LSE is found in 
\citet{GP83}; also see \citet{vdG93}. This result, however, holds only
in an asymptotic sense and only when the true function is a constant. Results on the pointwise adaptive behavior of the LSE in
monotone function estimation are more prevalent and can be
found, for example, in~\citet{CD99}, \citet{HW12} and \citet{Cator11}. For convex
function estimation, as far as we are aware, adaptation behavior of
the LSE has not been studied before. Adaptation behavior for
the estimation of a convex function at a single point has
been recently studied by~\citet{CaiLowFwork} but they focus on
different estimators that are based on local averaging techniques. 


\section{Risk Analysis of the LSE}\label{RiskAna}
Before stating our main risk bound, we need some notation. 
Recall that $\C$ denotes the set of all real-valued convex functions
on $[0, 1]$. For $\phi \in \C$, let $\lin(\phi)$ denote the
``distance'' of $\phi$ from affine functions. More precisely,
\begin{equation*}
\lin(\phi) := \inf \left\{\ell(\phi, \tau) : \tau \text{ is
      affine on } [0, 1] \right\}. 
\end{equation*}
Note that $\lin(\phi) = 0$ when $\phi$ is affine. 

We also need the notion of piecewise affine convex functions. A
convex function $\alpha$ on $[0, 1]$ is said to be piecewise 
affine if there exists an integer $k$ and points $0 = t_0 < t_1 <
\dots < t_k = 1$ such that $\alpha$ is affine on each of the $k$
intervals $[t_{i-1}, t_i]$ for $i = 1, \dots, k$. We define
$k(\alpha)$ to be the smallest such $k$. Let $\Ps_{k}$ denote the
collection of all piecewise affine convex functions with $k(\alpha)
\leq k$ and let $\Ps$ denote the collection of all piecewise affine
convex functions on $[0, 1]$.  

We are now ready to state our main upper bound for the risk of
$\hat{\phi}_{ls}$.  
\begin{thm}\label{twinf}
  Let $R := \max(1, \lin(\phi_0))$. There exists a positive constant
  $C$ depending only on the ratio $c_1/c_2$ such that 
  \begin{equation*}
    \E_{\phi_0} \ell^2(\hat{\phi}_{ls}, \phi_0)  \leq C \left(\log
      \frac{en}{2c_1} \right)^{5/4} \min \left[\left(\frac{\sigma^2
          \sqrt{R}}{n} \right)^{4/5}, \inf_{\alpha \in \Ps}
      \left(\ell^2(\phi_0, \alpha) + \frac{\sigma^2 k^{5/4}(\alpha)}{n}
      \right) \right] 
  \end{equation*}
  provided 
  \begin{equation*}
    n \geq C \frac{\sigma^2}{R^2} \left(\log
        \frac{en}{2c_1} \right)^{5/4}.  
  \end{equation*}
\end{thm}
Because of the presence of the minimum in the risk bound presented
above, the bound actually involves two  parts. We isolate these two
parts in the following two separate results. The first result
says that the risk is bounded by $n^{-4/5}$ up to
multiplicative factors that are logarithmic in $n$. The second result
says that the risk is bounded from above by a combination of the
parametric rate $1/n$ and an approximation term that dictates how well
$\phi_0$ is approximated by a piecewise affine convex function (up to
logarithmic multiplicative factors). The 
implications of these two theorems are explained in the remarks
below. It is clear that Theorem~\ref{rig1d} and~\ref{adap} together
imply  Theorem~\ref{twinf}. We therefore prove Theorem~\ref{twinf} by
proving Theorems~\ref{rig1d} and~\ref{adap} separately in 
Section~\ref{prbs}.  
\begin{thm}\label{rig1d}
  Let $R := \max(1, \lin(\phi_0))$. There exists a positive constant
  $C$ depending only on the ratio $c_1/c_2$ such that  
  \begin{equation*}
    \E_{\phi_0} \ell^2 \left(\hat{\phi}_{ls}, \phi_0 \right) \leq C
    \left(\log \frac{en}{2c_1} \right) \left(\frac{\sigma^2
        \sqrt{R}}{n} \right)^{4/5}   
  \end{equation*}
whenever 
\begin{equation*}
  n \geq C \left(\log \frac{en}{2c_1} \right)^{5/4}
  \frac{\sigma^2}{R^2}. 
\end{equation*}
\end{thm}
\begin{thm}\label{adap}
There exists a constant $C$, depending only on the ratio $c_1/c_2$, such
that 
\begin{equation}\label{adap.eq}
  \E_{\phi_0} \ell^2(\phi_0, \hat{\phi}_{ls}) \leq C \left(\log
    \frac{en}{2c_1} \right)^{5/4} \inf_{\alpha \in \Ps}
  \left(\ell^2(\phi_0, \alpha) + \frac{\sigma^2 k^{5/4}(\alpha)}n{}
  \right)  
\end{equation}
for all $n$. 
\end{thm}
The following remarks will better clarify the meaning of these
results. The first remark below is about Theorem~\ref{rig1d}. The later three remarks are about Theorem~\ref{adap}. 
%


\begin{remark}[Why convexity is similar to second order smoothness]\label{csi}
From the classical theory of nonparametric
statistics, it follows that this is the same rate that one obtains for
the estimation of twice differentiable functions (satisfying a
condition such as $\sup_{x \in [0, 1]} |\phi_0''(x)| \leq B$) on the
unit interval. In Theorem~\ref{rig1d}, we prove that
$\hat{\phi}_{ls}$ achieves the same rate (up to log factors) when the
true function is convex under no assumptions whatsoever on the
smoothness of the function. Therefore, the constraint of convexity is similar to the constraint of second order smoothness. This has long since been believed to be true, but to the best of our knowledge, Theorem~\ref{rig1d} is the first result to rigorously prove this  via a nonasymptotic risk bound for the estimator $\hat{\phi}_{ls}$ with no assumption of smoothness. 
\end{remark}

\begin{remark}[Parametric rates for piecewise affine convex functions]
  Theorem~\ref{adap} implies that $\hat{\phi}_{ls}$ has the parametric
  rate for estimating piecewise affine convex functions. Indeed,
  suppose $\phi_0$ is a piecewise affine convex function on $[0, 
1]$ i.e., $\phi_0 \in \Ps$. Then using $\alpha = \phi_0$
in~\eqref{adap.eq}, we have the risk bound 
\begin{equation*}
  \E_{\phi_0} \ell^2(\phi_0, \hat{\phi}_{ls}) \leq C
   \left(\log \frac{en}{2c_1} \right)^{5/4}
  \frac{\sigma^2 k^{5/4}(\phi_0)}{n}. 
\end{equation*}
This is the parametric rate $1/n$ up to logarithmic factors and is of
course much smaller than the nonparametric rate $n^{-4/5}$ given in
Theorem~\ref{rig1d}. Therefore, $\hat{\phi}_{ls}$ adapts to each class
$\Ps_k$ of piecewise convex affine functions.  
\end{remark}

\begin{remark}[Automatic adaptation]
Risk bounds such as~\eqref{adap.eq} are usually provable for
estimators based on empirical model selection criteria (see, for
example,~\citet{BarronBirgeMassart}) or aggregation (see, for
example,~\citet{RT12}). Specializing to 
the present situation, in order to adapt over $\Ps_k$ as $k$ varies,
one constructs LSE over each $\Ps_k$ and then either selects one
estimator from this collection by an  empirical model selection
criterion or aggregates these estimators with data-dependent weights. 
While the theory for such penalization
estimators is well-developed (see e.g.,~\citet{BarronBirgeMassart}),
these estimators are  computationally expensive, might rely on
certain tuning parameters which might be difficult to choose in
practice and also require estimation of $\sigma^2$. The LSE
$\hat{\phi}_{ls}$ is very different from these 
estimators because it simply minimizes the LS criterion
over the whole space $\C$. It is therefore very easy to compute, does
not depend on any tuning parameter or estimates for $\sigma^2$ and,
remarkably, it automatically adapts over the classes $\Ps_k$ as $k$
varies.      
\end{remark}

\begin{remark}[Why convexity is different from second order
  smoothness] In Remark~\ref{csi}, we argued how estimation under
  convexity is similar to estimation under second order
  smoothness. Here we describe how the two are different. The risk
  bound given by Theorem~\ref{adap} crucially depends on the true
  function $\phi_0$. In other words, the LSE converges at different
  rates  depending on the true convex function $\phi_0$. Therefore,
  the rate  of the LSE is not uniform over the class of all convex
  functions but  it varies quite a bit from function to function in
  that class. As will be clear from our proofs, the reason for this
  difference in rates is that the class of convex functions $\C$ is
  locally non-uniform in the sense that the local neighborhoods around
  certain convex functions (e.g., affine functions) are much sparser
  than local neighborhoods around other convex functions. On the other
  hand, in the class of twice differentiable functions, all local
  neighborhoods are, in some sense, equally sized.   
\end{remark}

\begin{remark}[On the logarithmic factors]
  We believe that Theorems~\ref{rig1d} and~\ref{adap} might have
  redundant logarithmic factors. In particular, we conjecture that there
  should be no logarithmic term in Theorem~\ref{rig1d} and that the
  logarithmic term should be $\log (en/(2c_1))$ instead of $(\log 
  (en/(2c_1)))^{5/4}$ in Theorem~\ref{adap}; cf. analogous results in
  isotonic regression -- \citet{Zhang02} and~\citet{CGS13}. These
  additional logarithmic factors mainly arise due to the fact that the
  class $S(\phi_0, r)$, of convex functions appearing in the proofs,
  is not uniformly bounded. Sharpening 
  these factors might be possible by using an
  explicit characterization of the LSE (as was done
  in~\citet{Zhang02} and~\citet{CGS13} for isotonic regression) and
  other techniques that are beyond the scope of the present paper.
\end{remark}

The proofs of Theorems~\ref{rig1d} and~\ref{adap} are presented in
Section~\ref{prbs}. A high level overview of the proof goes as
follows. The convex LSE is an ERM procedure. These procedures are very
well studied and numerous risk bounds exist in mathematical statistics
and machine learning (see, for example,~\citet{VandegeerBook, BM93,
  vaartwellner96book, Massart03Flour}). These results essentially say
that the risk behavior of $\hat{\phi}_{ls}$ is determined by the metric
entropy of the balls $S(\phi_0, r)$ (defined in~\eqref{lb}) in $\C$
around the true function $\phi_0$. Controlling the metric entropy of
the $S(\phi_0, r)$ is the key step in the proofs of
Theorem~\ref{rig1d} and~\ref{adap}. The next section deals with bounds
for the metric entropy of $S(\phi_0, r)$. 


\section{The Local Structure of the Space of Convex
  Functions}\label{MetricEnt} 
In this section, we prove bounds for the metric entropy of the balls
$S(\phi_0, r)$ as $\phi_0$ ranges over the space of convex
functions. Our results give new insights into the local structure of
the space of convex functions. We show that the metric entropy of
$S(\phi_0, r)$ behaves differently for different convex functions
$\phi_0$. This is the reason why the LSE exhibits different rates of
convergence depending on the true function $\phi_0$. The metric
entropy of $S(\phi_0, r)$ is much smaller when $\phi_0$ is a piecewise
affine  convex function with not too many affine pieces than when
$\phi_0$ has a second derivative that is bounded from above and below
by positive constants.  

The next theorem is the main result of this section. 
\begin{thm}\label{lball}
  There exists a positive constant $c$ depending
  only on the ratio $c_1/c_2$ such that for every $\phi_0 \in \C$ and
  $\epsilon > 0$, we have 
  \begin{equation}\label{lball.eq}
    \log M(\epsilon, S(\phi_0, r), \ell) \leq c  
  \left(\log \frac{en}{2c_1} \right)^{5/4} \sqrt{\frac{\Gamma(r;
      \phi_0)}{\epsilon}} 
\end{equation}
where 
\begin{equation*}
  \Gamma(r; \phi_0) :=  \inf_{\alpha \in \Ps} \left( k^{5/2}(\alpha)
    \left(r^2 + \ell^2(\phi_0, \alpha) \right)^{1/2} \right). 
\end{equation*}
\end{thm}
Note that the dependence of the right hand side on~\eqref{lball.eq} on 
$\epsilon$ is always $\epsilon^{-1/2}$. The dependence on $r$ is
given by $\Gamma(r; \phi_0)$ and it depends on $\phi_0$. This function
$\Gamma(r; \phi_0)$ controls the size of the ball $S(\phi_0, r)$. The
larger the value $\Gamma(r; \phi_0)$, the larger the metric entropy of
$S(\phi_0, r)$. The smallest possible value of $\Gamma(r; \phi_0)$
equals $r$ and is achieved for affine functions. When $\phi_0$ is
piecewise affine, $\Gamma(r; \phi_0)$ is larger than $r$ but it is not
much larger provided $k(\phi_0)$ is small. This is because $\Gamma(r;
\phi_0) \leq r k^{5/2}(\phi_0)$. When
$\phi_0$ cannot be well-approximable by piecewise affine functions
with small number of pieces, it can be shown that $\Gamma(r; \phi_0)$
is bounded from below by a constant independent of $r$. This will be
the case, for example, when $\phi_0$ is twice differentiable with
$\phi_0''(x)$ bounded from above and below by positive constants. As
shown in the next theorem, $S(\phi_0, r)$ has the largest possible
size for such $\phi_0$. Note also that one always has the upper bound
$\Gamma(r; \phi_0) \leq \sqrt{r^2 + \lin^2(\phi_0)}$ which can be
proved by restricting the infimum in the definition of $\Gamma(r;
\phi_0)$ to affine functions. 

We need the following definition for the next theorem. For a
subinterval $[a, b]$ of $[0, 1]$ and positive real numbers $\kappa_1 <
\kappa_2$, we define $\cur := \cur(a, b, \kappa_1, \kappa_2)$ to be
the class of all convex functions $\phi$ on $[0, 1]$ which are twice
differentiable on $[a, b]$ and which satisfy $\kappa_1 \leq \phi''(x)
\leq \kappa_2$ for all $x \in [a, b]$. 
\begin{thm}\label{lcur}
Suppose $\phi_0 \in \cur(a, b, \kappa_1, \kappa_2)$. Then
there exist positive constants 
$c$, $\epsilon_0$ and $\epsilon_1$ depending only on $\kappa_1,
\kappa_2$, $b-a$ and $c_2$ such that   
\begin{equation}\label{lcur.eq}
  \log M(\epsilon, S(\phi_0, r), \ell) \geq c \epsilon^{-1/2} \qt{for
    $\epsilon_1 n^{-2} \leq \epsilon \leq r \epsilon_0$}. 
\end{equation}
\end{thm}
Note that the right hand side of~\eqref{lcur.eq} does not depend on
$r$. This should be contrasted with the right hand side
of~\eqref{lball.eq} when $\phi_0$ is, say, an affine function. The
non-uniform nature of the space of univariate convex functions should
be clear from this: balls $S(\phi_0, r)$ of the same radius $r$ in the
space have different sizes depending on their center, $\phi_0$. This 
should be contrasted with the space of twice differentiable functions
in which all balls are equally sized in the sense that they all
satisfy~\eqref{lcur.eq}.   
\begin{remark}
  Note that the inequality~\eqref{lcur.eq} only holds when $\epsilon
  \geq \epsilon_1 n^{-2}$. In other words, it does not hold when
  $\epsilon \downarrow 0$. This is actually inevitable because,
  ignoring the convexity of functions in $S(\phi_0, r)$, the metric
  entropy of $S(\phi_0, r)$ under $\ell$ cannot be larger than the
  metric entropy of the ball of radius $r$ in $\R^n$, which is bounded from
  above by $n \log (1 + (3r/\epsilon))$ (see e.g.,~\citet[Lemma
  4.1]{Pollard90Iowa}). Thus, as $\epsilon \downarrow 0$, the metric
  entropy of $S(\phi_0, r)$ becomes 
  logarithmic in $\epsilon$ as opposed to $\epsilon^{-1/2}$. Also note
  that inequality~\eqref{lcur.eq} only holds for $\epsilon \leq r
  \epsilon_0$. This also makes sense because the diameter of
  $S(\phi_0, r)$ in the metric $\ell$ equals $2r$ and, consequently,
  the left hand side of~\eqref{lcur.eq} equals zero for $\epsilon >
  2r$. Therefore, one cannot expect~\eqref{lcur.eq} to hold for all
  $\epsilon > 0$. 
\end{remark}
\begin{remark}
  The proof of Theorem~\ref{lcur} actually implies a conclusion
  stronger than~\eqref{lcur.eq}. Let $S'(\phi_0, r) := \left\{\phi \in
  \C : \sup_x |\phi(x) - \phi_0(x)| \leq r\right\}$. Clearly this is a
 smaller neighborhood of $\phi_0$ than $S(\phi_0, r)$ i.e., $S'(\phi_0,
r) \subseteq S(\phi_0, r)$. The proof of Theorem~\ref{lcur} shows that
the lower bound~\eqref{lcur.eq} also holds for $\log M(\epsilon,
S'(\phi_0, r), \ell)$. 
\end{remark}

In the reminder of this section, we provide the proofs of
Theorems~\ref{lball} and~\ref{lcur}. Let us start with the proof of
Theorem~\ref{lball}. Since functions in $S(\phi_0, r)$ are convex,
we need to analyze the covering numbers of subsets of convex
functions. There exist only two previous results
here.~\citet{Bronshtein76} proved covering numbers for classes of
convex functions that are uniformly bounded and uniformly Lipschitz
under the supremum metric. This result was extended by~\citet{Dryanov}
who dropped the uniform Lipschitz assumption (this result was further
extended by~\citet{GS13} to the multivariate case). Unfortunately, the
convex functions in $S(\phi_0, r)$ are not uniformly
bounded (they only satisfy a weaker integral-type constraint)  and
hence Dryanov's result cannot be used directly for proving
Theorem~\ref{lball}. Another difficulty is that we need covering
numbers under $\ell$ while the results in~\citet{Dryanov} are based on
integral $L_p$ metrics. 

Here is a high-level outline of the proof of Theorem~\ref{lball}. The
first step is to reduce the general problem to the case when $\phi_0
\equiv 0$. The result for $\phi_0 \equiv 0$ immediately implies the
result for all affine functions $\phi_0$. One can then
generalize to piecewise affine convex functions by repeating the
argument over each affine piece. Finally, the result is derived for
general $\phi_0$ by approximating $\phi_0$ by piecewise affine convex
functions. 

For $\phi_0 \equiv 0$, the class of convex functions under
consideration is $S(0, r)$. Unfortunately, functions in $S(0, r)$ are
not uniformly bounded; they only satisfy a weaker discrete $L^2$-type
boundedness constraint. We get around the lack of uniform boundedness
by noting that convexity and the $L^2$-constraint
imply that functions in $S(0, r)$ are uniformly bounded on
subintervals that are in the interior of $[x_1, x_n]$ (this is proved
via Lemma~\ref{1db}). We use this to partition the interval $[x_1,
x_n]$ into appropriate subintervals where Dryanov's metric entropy
result can be employed. We first carry out this argument for another class of
convex functions where the discrete $L^2$-constraint is replaced by an
integral $L^2$-constraint. From this result, we deduce the covering
numbers of $S(0, r)$ by using straightforward interpolation results
(Lemma~\ref{inter}). 

\subsection{Proof of Theorem~\ref{lball}}
 
\subsubsection{Reduction to the case when $\phi_0 \equiv 0$}\label{anee}
The first step is to note that it suffices to prove the theorem when
$\phi_0$ is the constant function equal to 0. For $\phi_0 \equiv 0$,
Theorem~\ref{lball} is equivalent the following statement: there
exists a constant $c > 0$, depending only on the ratio $c_1/c_2$, such
that  
\begin{equation}\label{zc}
  \log M(\epsilon, S(0, r), \ell) \leq c \left(\log \frac{en}{2c_1}
  \right)^{5/4} \left(\frac{\epsilon}{r} \right)^{-1/2} \qt{for all
    $\epsilon > 0$}. 
\end{equation}
Below, we prove Theorem~\ref{lball} assuming that~\eqref{zc} is
true. Let $\alpha \in \Ps_k$ be a piecewise affine function with
$k(\alpha) = k$. We shall show that 
  \begin{equation}\label{pff}
    \log M(\epsilon, S(\alpha, r), \ell) \leq c k^{5/4} \left(\log
      \frac{en}{2c_1} \right)^{5/4} \left(\frac{\epsilon}{r}
    \right)^{-1/2} \qt{for every $\epsilon > 0$}. 
  \end{equation}
This inequality immediately implies Theorem~\ref{lball}
 because for every $\phi_0, \phi \in \C$ and $\alpha \in \Ps$, we have 
 \begin{equation*}
   \ell^2(\phi, \alpha) \leq 2 \ell^2(\phi, \phi_0) + 2 \ell^2(\phi_0,
   \alpha)
 \end{equation*}
 by the trivial inequality $(a + b)^2 \leq 2 a^2 + 2 b^2$. This means
 that $\ell^2(\phi, \alpha) \leq 2 r^2 + 2 \ell^2(\phi_0, \alpha)$ for
 every $\phi \in S(\phi_0, r)$. Hence 
 \begin{equation*}
   M(\epsilon, S(\phi_0, r), \ell) \leq M(\epsilon, S(\alpha,
   \sqrt{2(r^2 + \ell^2(\phi_0, \alpha))}, \ell). 
 \end{equation*}
 This inequality and~\eqref{pff} together clearly
 imply~\eqref{lball.eq}. It suffices therefore to prove~\eqref{pff}. 

Suppose that
$\alpha$ is affine on each of the $k$ intervals $I_i = [t_{i-1}, t_i]$
for $i = 2, \dots, k$, where $0 = t_0 < t_1 < \dots < t_{k-1} < t_k =
1$, and $I_1 = [0, t_1]$. Then there exist $k$ affine functions
$\tau_1, \dots, \tau_k$ on $[0, 1]$ such that $\alpha(x) = \tau_i(x)$
for $x \in I_i$ for every $i = 1, \dots, k$. 

For every pair of functions $f$ and $g$ on $[0, 1]$, we have the
trivial identity: $\ell^2(f, g) = \sum_{i=1}^k \ell_{i}^2(f,
g)$ where
\begin{equation*}
  \ell_i^2(f, g) := \frac{1}{n} \sum_{j: x_j \in I_i} \left(f(x_j) -
    g(x_j) \right)^2.  
\end{equation*}
  As a result, we clearly have 
\begin{equation}\label{wch}
  M(\epsilon, S(\alpha, r), \ell) \leq \prod_{i=1}^k
  M(\epsilon/\sqrt{k}, S(\alpha, r), \ell_{i}).
\end{equation}
Fix an $i \in \{1, \dots, k\}$. Note that for every $f \in S(\alpha,
r)$, we have 
\begin{equation*}
  \ell_i^2(f, \tau_i) = \ell_i^2(f, \alpha) \leq \ell^2(f, \alpha)
  \leq r^2. 
\end{equation*}
Therefore
\begin{equation*}
  M(\epsilon/\sqrt{k}, S(\alpha, r), \ell_i) \leq M(\epsilon/\sqrt{k},
  S_i(\tau_i, r), \ell_i) 
\end{equation*}
where $S_i(\tau_i, r)$ consists of the class of all convex functions
$f : I_i \rightarrow \R$ for which $\ell_i^2(\tau_i, f) \leq r^2$. 

By the translation invariance of the Euclidean  distance and the fact
that $\phi - \tau$ is convex whenever $\phi$ is convex and $\tau$
is affine, it follows that 
\begin{equation*}
  M(\epsilon/\sqrt{k}, S_i(\tau_i, r), \ell_i) = M(\epsilon/\sqrt{k}, S_i(0,
  r), \ell_i)
\end{equation*}
where $S_i(0, r)$ is defined as the class of all convex functions $f:
I_i \rightarrow \R$ for which $\ell_i^2(0, f) \leq r^2$.

The covering number $M(\epsilon/\sqrt{k}, S_i(0, r), \ell_i)$ can be
easily bounded using~\eqref{zc} by the following scaling
argument. Let $J := \{j \in \{1, \dots, n\} : x_j \in I_i\}$ with $m$
being the cardinality of $J$. Also write $[a, b]$ for the interval
$I_i$ and let $u_j := (x_j - a)/(b-a)$ for $j \in J$. For $f, g \in
\C$, let 
\begin{equation*}
  \ell^{(u)}(f, g) := \left( \frac{1}{m} \sum_{j \in J} (f(u_j) -
    g(u_j))^2 \right)^{1/2}
\end{equation*}
and $S^{(u)}(0, \gamma) := \{f \in \C : \ell^{(u)}(f, 0) \leq
\gamma\}$. By associating, for each $f \in S_i(0, r)$, the convex
function $\tilde{f} \in \C$ defined by $\tilde{f}(x) := f(a +
(b-a)x)$, it can be shown that 
\begin{equation*}
  M(\epsilon/\sqrt{k}, S_i(0, r), \ell_i) = M \left(\sqrt{\frac{n}{m}}
  \frac{\epsilon}{\sqrt{k}}, S^{(u)}(0, r \sqrt{n/m}), \ell^{(u)}
\right). 
\end{equation*}
The assumption~\eqref{eq:DesignPts} implies that the distance between
neighboring points in $\{u_j, j \in J\}$ lies between $mc_1/(n(b-a))$
and $mc_2/(n(b-a))$. Therefore, by applying~\eqref{zc} to $\{u_j, j
\in J\}$ instead of $\{x_i\}$, we obtain the existence of a positive
constant $c$ depending only on the ratio $c_1/c_2$ 
such that  
\begin{align*}
  \log M \left(\sqrt{\frac{n}{m}}
  \frac{\epsilon}{\sqrt{k}}, S^{(u)}(0, r \sqrt{n/m}), \ell^{(u)}
\right)  &\leq c \left(\log
    \frac{en(b-a)}{2c_1} \right)^{5/4} \left(\frac{\epsilon}{\sqrt{k}r}
  \right)^{-1/2} \\
&\leq c \left(\log
    \frac{en}{2c_1} \right)^{5/4} \left(\frac{\epsilon}{\sqrt{k}r}
  \right)^{-1/2}. 
\end{align*}
The required inequality~\eqref{pff} now follows from the above 
and~\eqref{wch}.   

\subsubsection{The Integral Version}
We have established above that it suffices to prove
Theorem~\ref{lball}  for $\phi_0 \equiv 0$ i.e., it suffices to
prove~\eqref{zc}. The ball $S(0, r)$ consists of all convex functions
$\phi$ such that 
\begin{equation}\label{cco}
  \frac{1}{n} \sum_{i=1}^n \phi^2(x_i) \leq r^2. 
\end{equation}
For $a < b$ and $B > 0$, let $\ic([a, b], B)$ denote the class of all
real-valued convex functions $f$ on $[a, b]$ for which $\int_a^b
f^2(x) dx \leq B^2$. The ball $S(0, r)$ is intuitively very close to
the class $\ic([0, 1], r)$ the only difference being that the average
constraint~\eqref{cco} is replaced by the integral constraint
$\int_0^1 \phi^2(x) dx \leq r^2$ in $\ic([0, 1], r)$. We shall prove a
good upper bound for the metric entropy of $\ic([0, 1], r)$. The
metric entropy of $S(0, r)$ will then be derived as a consequence. 
\begin{thm}\label{intcov}
There exist a constant $c$ such that for every $0 <
\eta < 1/2$, $B > 0$ and $\epsilon > 0$, we have 
  \begin{equation}\label{intcov.eq}
    \log M \left(\epsilon, \ic([0, 1], B), L_2[\eta, 1 -
      \eta] \right) \leq c \left(\log \frac{e}{2\eta} \right)^{5/4}
    \left(\frac{\epsilon}{B} \right)^{-1/2} . 
  \end{equation}
where, by $L_2[\eta, 1-\eta]$, we mean the metric where the distance between
$f$ and $g$ is given by 
\begin{equation*}
 \left( \int_{\eta}^{1-\eta} \left(f(x) - g(x) \right)^2 dx
 \right)^{1/2}. 
\end{equation*}
\end{thm}
\begin{remark}
  We take the metric above to be $L_2[\eta, 1 - \eta]$ as
  opposed to $L_2[0, 1]$ because 
  \begin{equation}
    \label{vinf}
\log M \left(\epsilon, \ic([0, 1], B), L_2[0, 1] \right) =
\infty 
  \end{equation}
 To see this, take $f_j(t) = 2^{j/2} \max(0, 1 - 2^j t)$ for $t \in
 [0,1]$ and $j \geq 1$. It is then easy to check that $f_j \in \ic([0,
 1], B)$ for $B \geq 1/3$ and that $\int_0^1 (f_j - f_{j+1})^2 \geq c$
 for some positive constant $c$ which proves~\eqref{vinf}. The
 equality~\eqref{vinf} is  also the reason why the right hand side
 of~\eqref{intcov.eq}  approaches $\infty$ as $\eta \downarrow 0$. 
\end{remark}

The above theorem is a new result. If the constraint $\int_0^1
\phi^2(x) dx \leq B^2$ is replaced by the stronger constraint $\sup_{x
\in [0, 1]} |\phi(x)| \leq B$, then this has been proved
by~\citet{Dryanov}. Specifically,~\citet{Dryanov} considered 
the class $\C([a, b], B)$ consisting of all convex functions $f$ on
$[a, b]$ which satisfy $\sup_{x \in [a, b]}|f(x)| \leq B$ and proved
the following.~\citet{GS13} extended this to the multivariate case.  
\begin{thm}[Dryanov]\label{IEEE} 
  There exists a positive constant $c$ such that for
  every $B > 0$ and $b > a$, we have 
  \begin{equation}\label{IEEE.eq}
    \log M \left(\epsilon, \C([a, b], B), L_2[a,b] \right) \leq c \left(\frac{\epsilon}{B (b -
        a)^{1/2}} \right)^{-1/2} \qt{for every $\epsilon > 0$}. 
  \end{equation}
\end{thm}
\begin{remark}
In~\citet{Dryanov}, inequality~\eqref{IEEE.eq} was only asserted for
$\epsilon \leq \epsilon_0 B(b-a)^{1/2}$ for a positive constant
$\epsilon_0$. It turns out however that this condition is
redundant. This follows from the observation that the diameter of the
space $\C([a, b], B)$ in the $L_2[a, b]$ metric is at most
$2B(b-a)^{1/2}$ which means that the left hand side of~\eqref{IEEE.eq}
equals 0 for $\epsilon > 2B(b-a)^{1/2}$ and, thus, by changing the
constant $c$ suitably in Dryanov's result, we obtain~\eqref{IEEE.eq}. 
\end{remark}

The class $\ic([0, 1], B)$ is much larger than $\C([0, 1], B)$ because
the integral constraint $\int_0^1 \phi^2(x) dx \leq B^2$ is much
weaker than $\sup_{x \in [0, 1]} |\phi(x)| \leq B$. Therefore,
Theorem~\ref{intcov} does not directly follow from
Theorem~\ref{IEEE}. However, it is possible to derive
Theorem~\ref{IEEE} from Theorem~\ref{intcov} via the observation (made
rigorous in Lemma~\ref{1db}) that functions in $\ic([0, 1], B)$ become 
uniformly bounded on subintervals of $[0, 1]$ that are sufficiently
far away from the boundary points. On such subintervals, we may use
Theorem~\ref{IEEE} to bound the covering numbers. Theorem~\ref{intcov}
is then proved by putting together these different covering numbers as
shown below.  
\begin{proof}[Proof of Theorem~\ref{intcov}]
By a trivial scaling argument, we can assume without loss of
generality that $B = 1$. Let $l$ be the largest integer that is
strictly smaller than $-\log(2\eta)/\log 2$ and let $\eta_i := 2^i
\eta$ for $i = 0, \dots, l+1$. Observe that $\eta_l < 1/2 \leq
\eta_{l+1}$. 

Fix $i \in \{0, \dots, l\}$. By Lemma~\ref{1db}, the restriction of a
function $\phi \in \ic([0, 1], 1)$ to $[\eta_i, 
\eta_{i+1}]$ is convex and uniformly bounded by $2 \sqrt{3}
\eta_i^{-1/2}$. Therefore, by Theorem~\ref{IEEE},
there exists a positive constant $c$ such that we
can cover the functions in $\ic([0, 1], 1)$ in the $L_2[\eta_i,
\eta_{i+1}]$ metric to within $\alpha_{i}$ by a finite set having
cardinality at most 
\begin{equation*}
  \exp \left[ c \left(\frac{\alpha_i \sqrt{\eta_i}}{\sqrt{\eta_{i+1} -
        \eta_i}} \right)^{-1/2} \right] = \exp \left(c \alpha_i^{-1/2}
\right). 
\end{equation*}
Because
  \begin{equation*}
    \int_{\eta}^{1/2} \left(\phi(x) - f(x) \right)^2 dx \leq
    \sum_{i=0}^{l} \int_{\eta_i}^{\eta_{i+1}}
       \left(\phi(x) - f(x) \right)^2 dx, 
  \end{equation*}
we get a cover for functions in $\ic([0, 1], 1)$ in the
$L_2[\eta, 1/2]$ metric of size less than or equal to $\left(\sum_{i=0}^l
\alpha_i^2 \right)^{1/2}$ and cardinality at most $\exp \left(c
\sum_{i=0}^l \alpha_i^{-1/2} \right)$. 

Taking $\alpha_i = \epsilon (l+1)^{-1/2}$, we get that
\begin{equation*}
  \log M(\epsilon, \ic([0, 1], 1), L_2[\eta, 1/2]) \leq  c
    \epsilon^{-1/2} (l+1)^{5/4} \leq c_1 \epsilon^{-1/2} \left(\log
      \frac{e}{2\eta} \right)^{5/4}
\end{equation*}
where $c_1$ depends only on $c$. By an analogous argument, the above
inequality will also hold for  $\log M(\epsilon, \ic([0, 1], 1),
L_2[1/2, 1-\eta])$. The proof is completed by putting these two bounds
together.   
\end{proof}

\subsubsection{Completion of the Proof of Theorem~\ref{lball}}
We now complete the proof of Theorem~\ref{lball} by proving
inequality~\eqref{zc}. We will use Theorem~\ref{intcov}. We need to
switch between the pseudometrics $\ell$ and $L_2[\eta, 1-\eta]$. This
will be made convenient by the use of Lemma~\ref{inter}.  

By an elementary scaling argument, it follows that 
  \begin{equation*}
    M(\epsilon, S(0, r), \ell) = M(\epsilon/r, S(0, 1), \ell). 
  \end{equation*}
We, therefore, only need to prove~\eqref{zc} for $r = 1$. For ease of
notation, let us denote $S(0, 1)$ by $S$.   
  
Because $x_i - x_{i-1} \geq c_1/n$ for all $i = 2, \dots, n$,  we have
$x_2, \dots, x_{n-1} \in [c_1/n, 1 - 
(c_1/n)]$. We shall first 
  prove an upper bound for $\log M(\epsilon, S, \ell_1)$ where  
  \begin{equation*}
    \ell^2_1(\phi, \psi) := \frac{1}{n-2} \sum_{i=2}^{n-1}
    \left(\phi(x_i) - \psi(x_i) \right)^2. 
  \end{equation*}
  For each function $\phi \in S$, let $\tilde{\phi}$ be the convex function
  on $[x_2, x_{n-1}]$ defined by 
  \begin{equation*}
    \tilde{\phi}(x) := \frac{x_{i+1} - x}{x_{i+1} - x_i}\phi(x_i)  +
  \frac{x-x_i}{x_{i+1} - x_i} \phi(x_{i+1}) \qt{for $x_i \leq x \leq
    x_{i+1}$} 
  \end{equation*}
  where $i = 2, \dots, n-2$. Also let $\tilde{S} :=
  \left\{\tilde{\phi}: \phi \in S \right\}$. 

By Lemma~\ref{inter} and the assumption
  that $x_i - x_{i-1} \geq c_1/n$ for all $i$, we get that 
  \begin{equation*}
    \ell_1^2(\phi, \psi) \leq \frac{6}{c_1} \int_{x_2}^{x_{n-1}}
    \left(\tilde{\phi}(x) - \tilde{\psi}(x) \right)^2 dx
\end{equation*}
for every pair of functions $\phi$ and $\psi$ in $S$. Letting $\delta
:= \epsilon \sqrt{c_1/6}$ this inequality implies that 
  \begin{equation*}
    M \left(\epsilon, S, \ell_1 \right) \leq  M \left(\delta, \tilde{S}, L_2[x_2, x_{n-1}] \right).  
  \end{equation*}
 Again by Lemma~\ref{inter} and the assumption $x_i - x_{i-1} \leq
 c_2/n$, we have that 
 \begin{equation*}
   \int_{x_1}^{x_n} \tilde{\phi}^2(x) dx \leq \frac{c_2}{n}
   \sum_{i=1}^n \phi^2(x_i) \leq c_2 \qt{for every $\phi \in S$}. 
 \end{equation*}
As a result, we have that $\tilde{S} \subseteq \ic([x_1, x_n],
\sqrt{c_2})$. Further, because $x_2 \geq x_1 + c_1/n$ and $x_{n-1} \leq x_n -
c_1/n$, we get that 
\begin{equation*}
  M \left(\delta, \tilde{S}, L_2[x_2, x_{n-1}] \right) \leq M
    \left(\delta, \ic([x_1, x_n], \sqrt{c_2}), L_2[x_1 + \eta, x_{n} - \eta]
    \right)
\end{equation*}
where $\eta := c_1/n$. By a simple scaling argument, the covering
number on the right hand side above is upper bounded by  
\begin{equation}\label{ari}
 M \left(\frac{\delta}{\sqrt{x_n - x_1}}, \ic([0, 1], \sqrt{c_2(x_n -
     x_1)}), L_2\left[\frac{\eta}{x_n - x_1}, 1 - \frac{\eta}{x_n - 
    x_1}\right]  \right). 
\end{equation}
Indeed, for each $f \in \ic([x_1, x_n], \sqrt{c_2})$, we can associate
$\tilde{f}(y) := f(x_1 + y(x_n - x_1))$ for $y \in [0, 1]$. It is then
easy to check that $\tilde{f} \in \ic([0, 1], \sqrt{c_2(x_n - x_1)})$
and 
\begin{equation*}
  \int_{x_1 + \eta}^{x_n - \eta} \left(f_1(x) - f_2(x) \right)^2 dx =
  (x_n - x_1) \int_{\eta/(x_n - x_1)}^{1 - (\eta/(x_n - x_1))}
  \left(\tilde{f}_1(y) - \tilde{f}_2(y) \right)^2 dy,
\end{equation*}
from which~\eqref{ari} easily follows. From the bound~\eqref{ari}, it
is now easy to see that  (because $x_n - x_1 \leq 1$)   
\begin{equation*}
  M \left(\delta, \ic([x_1, x_n], \sqrt{c_2}), L_2[x_1 + \eta, x_{n} - \eta]
  \right)  \leq   M \left(\delta, \ic([0, 1], \sqrt{c_2}), L_2[\eta, 1 - \eta] 
  \right). 
\end{equation*}
Thus, by Theorem~\ref{intcov}, we assert the existence of a positive
constant $c$ a such that 
\begin{equation}\label{keyc}
\log  M(\epsilon, S, \ell_1) \leq c \left(\log \frac{en}{2c_1}
\right)^{5/4} \left(\frac{\sqrt{c_1} \epsilon}{\sqrt{c_2}} \right)^{-1/2}. 
\end{equation}
Now for every pair of functions $\phi$ and $\psi$ in $S$, we have 
\begin{equation*}
  \ell^2(\psi, \phi) \leq \ell_1^2(\psi, \phi) + \frac{1}{n} \sum_{i
    \in \{1, n\}} \left(\phi(x_i) - \psi(x_i) \right)^2.
\end{equation*}
We make the simple observation that $(\phi(x_1), \phi(x_n))$ lies in
the closed ball of radius $\sqrt{n}$ in $\R^2$ denoted by $B_2(0,
\sqrt{n})$. As a result, using ~\citet[Lemma
  4.1]{Pollard90Iowa}, we have
\begin{equation*}
  M(\epsilon, S, \ell) \leq M(\frac{\epsilon}{\sqrt{2}}, S, \ell_1)
  M(\frac{\sqrt{n}\epsilon}{\sqrt{2}}, B_2(0, \sqrt{n})) \leq \left(1
    + \frac{3\sqrt{2}}{\epsilon} \right)^2
  M(\frac{\epsilon}{\sqrt{2}}, S, \ell_1) 
\end{equation*}
where the covering number of $B_2(0, \sqrt{n})$ is in the usual
Euclidean metric. Using~\eqref{keyc}, we get 
\begin{equation}\label{au}
  \log M(\epsilon, S, \ell) \leq 2 \log \left(1 +
    \frac{3\sqrt{2}}{\epsilon} \right) + c \left(\log
    \frac{en}{2c_1} \right)^{5/4} \left(\frac{\sqrt{c_1}
      \epsilon}{\sqrt{2c_2}} \right)^{-1/2} .  
\end{equation}
Because $\log(1 + x) \leq 3 \sqrt{x}$ for all $x > 0$, the first term
in the right hand side above is bounded by a constant multiple of
$\epsilon^{-1/2}$. This proves~\eqref{zc} provided the constant $c$ is
renamed appropriately. 

\subsection{Proof of Theorem~\ref{lcur}}
In our proof below, we shall make use of Lemma~\ref{buj} (stated and
proved in Section~\ref{AuxRes}) which bounds the distance between
functions in $\cur(a, b, \kappa_1, \kappa_2)$ and their piecewise
linear interpolants.  

Fix $m \geq 1$ and let $t_i = a + (b-a)i/m$ for $i = 0, \dots,
  m$. For each $i = 1, \dots, m$, let $\alpha_i$ define the linear
  interpolant of the points $(t_{i-1}, \phi_0(t_{i-1}))$ and $(t_i,
  \phi_0(t_i))$ i.e., 
  \begin{equation*}
    \alpha_i(x) := \phi_0(t_{i-1}) + \frac{\phi_0(t_i) -
      \phi_0(t_{i-1})}{t_i - t_{i-1}}  \left(x - t_{i-1} \right)
    \qt{for $x \in [0, 1]$}. 
  \end{equation*}
By error estimates for linear interpolation (see e.g., Chapter 3 of
\citet{Atkinson88}), for every $x \in [t_{i-1}, t_i]$, there exists a
point $t_x \in [t_{i-1}, t_i]$ for which 
  \begin{equation*}
    |\phi_0(x) - \alpha_i(x)| = (x - t_{i-1})(t_i - x)
    \frac{\phi_0''(t_x)}{2}
  \end{equation*}
  which implies, because $\phi_0 \in \cur(a, b, \kappa_1, \kappa_2)$,
  that 
  \begin{equation}\label{jabil}
 |\phi_0(x) - \alpha_i(x)| \leq (x - t_{i-1})(t_i - x)
 \frac{\kappa_2}{2} \leq \frac{\kappa_2}{8}(t_i - t_{i-1})^2 =
 \frac{(b-a)^2 \kappa_2}{8 m^2}
  \end{equation}
for every $x \in [a, b]$. By convexity of $\phi_0$, it is
obvious that $\alpha_i(x) \geq \phi_0(x)$ for $x \in [t_{i-1}, t_i]$
and $\alpha_i(x) \leq \phi_0(x)$ for $x \notin [t_{i-1}, t_i]$. 

Now for each $\tau \in \{0, 1\}^m$, let us define 
  \begin{equation*}
    \phi_{\tau}(x) := \max \left(\phi_0(x), \max_{i : \tau_i = 1}
      \alpha_i(x) \right)  \qt{for $x \in [0, 1]$}. 
  \end{equation*}
  The functions $\phi_{\tau}$ are clearly convex because they equal
  the pointwise maximum of convex functions. Moreover, for $x \in
  [t_{i-1}, t_i]$, we have
\begin{equation*}
\phi_{\tau}(x) = \left\{ \begin{array}{rl}
 \alpha_i(x) &\mbox{if $\tau_i = 1$} \\
 \phi_0(x) &\mbox{if $\tau_i = 0$.}
       \end{array} \right.
\end{equation*}
Also, from~\eqref{jabil}, 
\begin{equation*}
  \sup_{x \in [0, 1]} \left|\phi_{\tau}(x) - \phi_0(x) \right| \leq
  \max_{1 \leq i \leq m} \sup_{x \in [t_{i-1}, t_i]} \left|\phi_{0}(x)
    - \alpha_i(x) \right| \leq \frac{(b-a)^2\kappa_2}{8 m^2}. 
\end{equation*}
Because $\ell(\phi_{\tau}, \phi_0) \leq \sup_{x} |\phi_{\tau}(x) -
\phi_0(x)|$, it follows that $\phi_{\tau} \in S(\phi_0, r)$ provided 
\begin{equation}\label{bai}
  \frac{(b-a)^2 \kappa_2}{8 m^2} \leq r. 
\end{equation}
Observe now that for every $\tau, \tau' \in \{0, 1\}^m$,  
\begin{equation}\label{rsh}
  \ell^2\left(\phi_{\tau}, \phi_{\tau'} \right) = \sum_{i: \tau_i \neq
  \tau_i'} \ell^2 \left(\phi_0, \max(\phi_0, \alpha_i) \right) \geq
\ham(\tau, \tau') \min_{1 \leq i \leq m} \ell^2(\phi_0, \max(\phi_0,
\alpha_i)) 
\end{equation}
where $\ham(\tau, \tau') := \sum_i \{\tau_i \neq \tau'_i \}$. We now
use Lemma~\ref{buj} to bound $\ell^2(\phi_0, \max(\phi_0, \alpha_i))$
from below. Since $\alpha_i$ is the linear interpolant of $(t_{i-1},
\phi_0(t_{i-1}))$ and $(t_i, \phi_0(t_i))$, we use Lemma~\ref{buj}
(inequality~\eqref{ilo}) with $a = t_{i-1}$ and $b = t_i$ to assert  
\begin{equation*}
  \ell^2(\phi_0, \max(\phi_0, \alpha_i)) \geq \frac{\kappa_1^2 (t_i -
    t_{i-1})^5}{4096 c_2} = \frac{\kappa_1^2 (b-a)^5}{4096 c_2 m^5} 
\end{equation*}
provided 
\begin{equation}\label{mck}
  n \geq \frac{4c_2}{t_i-t_{i-1}} = \frac{4mc_2}{b-a}. 
\end{equation}
From~\eqref{rsh}, we thus have
\begin{equation*}
  \ell^2(\phi_{\tau}, \phi_{\tau'}) \geq \ham(\tau, \tau')
  \frac{\kappa_1^2 (b-a)^5}{4096 c_2 m^5} . 
\end{equation*}
Using now the Varshamov-Gilbert lemma (see, for example,~\citet[Lemma
4.7]{Massart03Flour}) which asserts the existence of a subset $W$ of
$\{0, 1\}^m$ with cardinality, $|W| \geq \exp(m/8)$ such that
$\ham(\tau, \tau') \geq m/4$ for all $\tau, \tau' \in W$ with $\tau
\neq \tau'$, we get that 
\begin{equation}\label{xwe}
  \ell^2(\phi_{\tau}, \phi_{\tau'}) \geq \frac{\kappa_1^2
    (b-a)^5}{16384 c_2 m^4} \qt{for all $\tau, \tau' \in W$ with $\tau
    \neq \tau'$}. 
\end{equation}
Let us now fix $\epsilon > 0$ and choose $m$ so that 
\begin{equation*}
  m^4 = \frac{\kappa_1^2(b-a)^5}{16384 c_2 \epsilon^2}. 
\end{equation*}
From~\eqref{xwe}, we then see that $\{\phi_{\tau}: \tau \in W\}$ is an
$\epsilon$-packing set under the pseudometric $\ell$. The
condition~\eqref{bai} would hold provided 
\begin{equation*}
  \epsilon \leq \frac{\kappa_1 \sqrt{b-a}}{16 \sqrt{c_2} \kappa_2} r. 
\end{equation*}
Also, the condition~\eqref{mck} is equivalent to 
\begin{equation*}
  \epsilon \geq \frac{c_2^2 \sqrt{b-a} \kappa_1}{8 \sqrt{c_2}n^2} . 
\end{equation*}
We have therefore showed that for $\epsilon$ satisfying the above
pair of inequalities, there exists an $\epsilon$-packing subset of
$S(\phi_0, r)$ with cardinality $|W|$ satisfying 
\begin{equation*}
 \log |W| \geq \frac{m}{8} \geq \frac{\sqrt{\kappa_1} (b-a)^{5/4}}{96
   c_2^{1/4}} \epsilon^{-1/2}.
\end{equation*}
The proof of Theorem~\ref{lcur} is now complete if we take 
\begin{equation*}
  \epsilon_0 := \frac{\kappa_1 \sqrt{b-a}}{16 \kappa_2 \sqrt{c_2}}
  \quad \text{ and } \quad c := \frac{\sqrt{\kappa_1} (b-a)^{5/4}}{96
    c_2^{1/4}} \quad \text{ and } \quad \epsilon_1 :=  \frac{c_2^2
    \sqrt{b-a} \kappa_1}{8 \sqrt{c_2}} .  
\end{equation*}
\section{Proofs of the Risk Bounds of the LSE}\label{prbs}
In this section, we provide the proofs of Theorems~\ref{rig1d}
and~\ref{adap}. As mentioned in Section~\ref{MetricEnt}, these two
theorems together imply our main risk bound Theorem~\ref{twinf} of the
convex LSE. Our proofs are based on the local metric entropy result
(Theorem~\ref{lball}) of the space of univariate convex functions 
derived in the previous section together with standard results on the
risk behavior of ERM procedures. Before proceeding further, let us state precisely the result from the literature on ERM procedures that we use to analyze the risk of $\hat{\phi}_{ls}$. There exist many such results but they are all similar in spirit and the following result from~\citet[Theorem
9.1]{VandegeerBook} is especially convenient to use.  
\begin{thm}[Van de Geer]\label{vandy}
For each $r > 0$, let $$S(\phi_0, r) := \{\phi \in \C: \ell^2(\phi_0,
\phi) \leq r^2 \}.$$ Suppose $H$ is a function on $(0, \infty)$ such that
  \begin{equation*}
    H(r) \geq \int_0^r \sqrt{\log M(\epsilon, S(\phi_0, r) , \ell)} 
    \ d\epsilon \qt{for every $r > 0$}
  \end{equation*}
and such that $H(r)/r^2$ is decreasing on $(0, \infty)$. Then there
exists a universal constant $C$  such that
\begin{equation*}
  \P_{\phi_0}\left(\ell^2(\hat{\phi}_{ls}, \phi_0) > \delta \right)
  \leq C \sum_{s \geq 0} \exp \left(- \frac{n2^{2s}\delta}{C^2
      \sigma^2} \right) 
\end{equation*}
for every $\delta > 0$ satisfying $\sqrt{n} \delta \geq C\sigma
H(\sqrt{\delta})$.  
\end{thm}
Let us note that our local metric entropy result, Theorem~\ref{lball},
easily implies an upper bound for the entropy integral 
  \begin{equation}\label{mavik}
    \int_0^r \sqrt{\log M(\epsilon, S(\phi_0, r), \ell)} d\epsilon
  \end{equation}
appearing in Theorem~\ref{vandy}. Indeed, using the bound given
by~\eqref{lball.eq} for $M(\epsilon, S(\phi_0, r), \ell)$ above and
integrating, we obtain that~\eqref{mavik} is bounded from above by 
  \begin{equation}\label{eib.eq}
    K \left(\log \frac{en}{2c_1} \right)^{5/8} r^{3/4}
    \inf_{\alpha \in \Ps} \left[ k^{5/8}(\alpha)  \left(r^2 +
        \ell^2(\phi_0, \alpha) \right)^{1/8} \right]
\end{equation}
  for every $\phi_0 \in \C$ and $r > 0$ where $K$ is a constant that
  only depends on the ratio $c_1/c_2$. 

\subsection{Proof of Theorem~\ref{rig1d}}\label{GblRates}
Let us define 
\begin{equation*}
  \delta_0 := A \left(\frac{\sigma^2}{n} \right)^{4/5}
  R^{2/5} \log \frac{en}{2c_1}
\end{equation*}
where $A$ is a constant whose value will be specified shortly. Observe
that $\delta_0 \leq R^2$ whenever  $n \geq A^{5/4} \left(\log
 ((en)/(2c_1)) \right)^{5/4} \sigma^2/R^2$. We use the
bound~\eqref{eib.eq} for the entropy 
integral~\eqref{mavik}. By restricting the infimum in 
the right hand side of~\eqref{eib.eq} to affine functions (i.e.,
$\alpha \in \Ps_1$) for which $k(\alpha) = 1$, we obtain (note that
$\inf_{\alpha \in \Ps_1} \ell^2(\phi_0, \alpha) = \lin^2(\phi_0)
\leq R^2$)
\begin{equation}\label{messi}
  \int_0^r \sqrt{\log M(\epsilon, S(\phi_0, r), \ell)} d\epsilon \leq
  K \left(\log \frac{en}{2c_1} \right)^{5/8} r^{3/4} \left(r^2 + R^2
  \right)^{1/8} 
\end{equation}
for every $r > 0$. Suppose now that 
\begin{equation}
  \label{ncon}
  n \geq A^{5/4} \left(\log \frac{en}{2c_1}
    \right)^{5/4} \frac{\sigma^2}{R^2}
\end{equation}
so that $\delta_0 \leq R^2$ and inequality~\eqref{messi} holds for
every $r > 0$. Let $H(r)$ denote the right hand side
of~\eqref{messi}. It is clear that $H(r)/r^2$ is decreasing on $(0,
\infty)$. As a  result, a condition of the form  $\sqrt{n} \delta \geq
C \sigma H(\sqrt{\delta})$ for some positive constant $C$ holds for
every $\delta \geq \delta_0$ provided it holds for $\delta =
\delta_0$. Clearly  
\begin{equation*}
  \frac{H(\sqrt{\delta_0})}{\delta_0} = K \left(\log \frac{en}{2c_1}
  \right)^{5/8} \delta_0^{-5/8} \left(\delta_0 + R^2 \right)^{1/8}.
\end{equation*}
Assuming that~\eqref{ncon} holds and noting then that $\delta_0 \leq
R^2$, we get  
\begin{equation*}
  \frac{H(\sqrt{\delta_0})}{\delta_0} \leq 2^{1/8} K \left(\log
    \frac{en}{2c_1} \right)^{5/8} \delta_0^{-5/8} R^{1/4} =
  2^{1/8} K A^{-5/8} \frac{\sqrt{n}}{\sigma}. 
\end{equation*}
We shall now use Theorem~\ref{vandy}. Let $C$ be the
constant given by  
Theorem~\ref{vandy}. By the above inequality, the condition $\sqrt{n}
\delta \geq C \sigma H(\sqrt{\delta})$ holds for each $\delta \geq
\delta_0$ provided $A = 2^{1/5} (C K)^{8/5}$. Thus 
by Theorem~\ref{vandy}, we obtain  
\begin{equation*}
  \P_{\phi_0} \left(\ell^2(\hat{\phi}_{ls}, \phi_0) > \delta \right)
  \leq C \sum_{s \geq 0} \exp \left(-\frac{n2^{2s} \delta}{C^2
      \sigma^2} 
  \right) 
\end{equation*}
for all $\delta \geq \delta_0$ whenever $n$
satisfies~\eqref{ncon}. Using the expression for $\delta_0$
and~\eqref{ncon}, we get  for $\delta \ge \delta_0$,
\begin{equation}\label{auxie}
  \frac{n\delta}{\sigma^2} \geq \frac{n \delta_0}{\sigma^2} = A
  \left(\frac{n}{\sigma^2} \right)^{1/5} R^{2/5}  \log
  \frac{en}{2c_1} \geq A^{5/4} \left(\log \frac{en}{2c_1} \right)^{5/4}. 
\end{equation}
We thus have
\begin{equation*}
  \P_{\phi_0} \left(\ell^2(\hat{\phi}_{ls}, \phi_0) > \delta \right)
  \leq C_1 \exp \left(- \frac{n \delta}{C_1 \sigma^2} \right) \qt{for
    all $\delta \geq \delta_0$}
\end{equation*}
for some constant $C_1$ (depending only on $C$ and $A =
2^{1/5}(CK)^{8/5}$) provided $n$ 
satisfies~\eqref{ncon}. Integrating both sides of this inequality
with respect to $\delta$ (and using~\eqref{auxie} again), we obtain
the risk bound 
\begin{equation*}
  \E_{\phi_0} \ell^2(\hat{\phi}_{ls}, \phi_0 ) \leq C_2 \delta_0 = C_2
  \left(\frac{\sigma^2}{n} \right)^{4/5} A R^{2/5} \log
  \frac{en}{2c_1} 
\end{equation*}
for some positive constant $C_2$ depending only on $C$ and
$K$. Because $C$ is an absolute constant and $K$ only depends on the
ratio $c_1/c_2$, the proof is complete by an appropriate renaming of
the constant $C$.  

\subsection{Proof of Theorem~\ref{adap}}\label{Adap}
For each $1 \leq k \leq n$, let $$\ell_k^2 = \inf \{\ell^2(\phi_0,
\alpha) : \alpha \in \Ps \text{ and } k(\alpha) = k \}$$ so that 
\begin{equation*}
  \inf_{\alpha \in \Ps} \left(\ell^2(\phi_0, \alpha) + \frac{\sigma^2
      k^{5/4}(\alpha)}{n} \right) = \inf_{1 \leq k \leq n}
  \left(\ell_k^2 + \frac{\sigma^2 k^{5/4}}{n} \right). 
\end{equation*}
It is also easy to check that 
\begin{equation*}
  \ell_1^2 \geq \ell_2^2 \geq \dots \geq \ell_n^2 = 0. 
\end{equation*}
As a result, there exists an integer $u \in \{1, \dots, n\}$ such that
$\ell_k^2 > \sigma^2 k^{5/4}/n$ if $1 \leq k < u$ and $\ell_k^2 \leq
\sigma^2 k^{5/4}/n$ if $k \geq u$. This means that when $1 \leq k < u$
(which implies that $u \geq 2$ or $u-1 \geq u/2$)
\begin{equation*}
  \ell_k^2 + \frac{\sigma^2 k^{5/4}}{n} \geq \ell_{u-1}^2 >
  \frac{\sigma^2}{n} (u-1)^{5/4} \geq \frac{\sigma^2
    u^{5/4}}{2^{5/4}n}. 
\end{equation*}
It then follows that
\begin{equation*}
  \inf_{1 \leq k \leq n} \left(\ell_k^2 + \frac{\sigma^2 k^{5/4}}{n}
  \right) \geq \frac{\sigma^2 u^{5/4}}{2^{5/4}n}. 
\end{equation*}
Consequently, the proof will be complete if we show that
\begin{equation}\label{peq}
  \E_{\phi_0} \ell^2(\phi_0, \hat{\phi}_{ls}) \leq C \left( \log
    \frac{en}{2c_1} \right)^{5/4} \frac{\sigma^2 u^{5/4}}{n}.
\end{equation}
To prove this, we start by defining 
  \begin{equation*}
    \delta_0 := A \left(\log \frac{en}{2c_1}
    \right)^{5/4} \frac{\sigma^2 u^{5/4}}{n}
  \end{equation*}
for a constant $A$ whose value will be specified shortly. Because
$\ell_u^2 \leq \sigma^2 u^{5/4}/n$, it follows that $\ell_u^2 \leq
\delta_0/A$. 

By~\eqref{eib.eq}, there exists a positive constant
  $K$ depending only on the ratio $c_1/c_2$ such that
  \begin{align*}
    \int_0^r \sqrt{\log M(\epsilon, S(\phi_0, r), \ell)} d\epsilon
    &\leq K \left(\log \frac{en}{2c_1} \right)^{5/8} \inf_{\alpha \in
      \Ps} \left[ k^{5/8}(\alpha)
    r^{3/4} \left(r^2 + \ell^2(\phi_0, \alpha) \right)^{1/8} \right]
  \\
&\leq K \left(\log \frac{en}{2c_1} \right)^{5/8} \inf_{\alpha \in
      \Ps_u} \left[ k^{5/8}(\alpha)
    r^{3/4} \left(r^2 + \ell^2(\phi_0, \alpha) \right)^{1/8} \right]
  \\
&\leq K \left(\log \frac{en}{2c_1} \right)^{5/8} u^{5/8} r^{3/4}
\left(r^2 + \ell_u^2 \right)^{1/8}. 
  \end{align*}
for every $r > 0$. Let $H(r)$ denote the right hand side above. It is
clear that $H(r)/r^2$ is decreasing on $(0, \infty)$. As a result, a
condition of the form $\sqrt{n} \delta \geq C \sigma H(\sqrt{\delta})$
for some positive constant $C$ holds for every $\delta \geq \delta_0$
provided it holds for $\delta = \delta_0$. Because $\ell_u^2 \leq
\delta_0/A$, we have
  \begin{equation*}
    H(\sqrt{\delta_0}) \leq K \left(\log \frac{en}{2c_1} \right)^{5/8}
    u^{5/8} \sqrt{\delta_0} \left(1 + \frac{1}{A} \right)^{1/8}. 
  \end{equation*}
  Consequently, 
  \begin{equation}\label{kal}
    \frac{H(\sqrt{\delta_0})}{\delta_0} \leq \frac{K}{\sqrt{A}}  \left(1 +
      \frac{1}{A} \right)^{1/8}
    \frac{\sqrt{n}}{\sigma}. 
  \end{equation}
  We shall now use Theorem~\ref{vandy}. Let $C$ be the positive 
  constant given by Theorem~\ref{vandy}. By inequality~\eqref{kal}, we
  can clearly choose $A$ depending only on $K$ and $C$ so that
  $\sqrt{n} \delta_0 \geq C \sigma H(\sqrt{\delta_0})$. Because
  $H(r)/r^2$ is a decreasing function of $r$, this choice of $A$ also
  ensures that $\sqrt{n} \delta \geq C \sigma H(\sqrt{\delta})$ for
  every $\delta \geq \delta_0$. Thus by Theorem~\ref{vandy}, we obtain 
  \begin{equation}\label{vep}
    \P_{\phi_0} \left(\ell^2(\hat{\phi}_{ls}, \phi_0) > \delta \right)
    \leq C \sum_{s \geq 0} \exp \left(- \frac{n2^{2s} \delta}{C^2
        \sigma^2} \right) \qt{for all $\delta \geq \delta_0$}. 
  \end{equation}
Note further, from the definition of $\delta_0$, that $\delta_0 \geq
\sigma^2 A/n$   which implies that the sum on the right hand side
of~\eqref{vep} is dominated by the first term. We thus have 
  \begin{equation*}
    \P_{\phi_0} \left(\ell^2(\hat{\phi}_{ls}, \phi_0) > \delta \right)
    \leq C_1 \exp \left(- \frac{n \delta}{C_1 \sigma^2} \right)
    \qt{for all $\delta \geq \delta_0$}.  
  \end{equation*}
  for a constant $C_1$ depending upon only $C$ and $A$. The required
  risk bound~\eqref{peq} is now derived by integrating both sides
  of the above inequality with respect to $\delta$ and using that
  $\delta_0 \geq \sigma^2 A/n$.

\section{Non-adaptable convex functions}\label{LowerBd}
We showed that the risk of the convex LSE is always bounded from above
by $n^{-4/5}$ up to logarithmic factors in $n$ and that for convex
functions that are well-approximable by piecewise affine functions
with not too many pieces, the risk of the convex LSE is bounded by
$1/n$ up to log factors. The reason why the risk is much smaller for
these functions is that the balls around them have small
sizes. We also showed in Theorem~\ref{lcur} that for convex functions
with curvature, the balls are really non-local. Here, we show
that for such convex functions, in a very strong sense, the rate
$n^{-4/5}$  cannot be improved by \textit{any} estimator. 

Recall the class of functions, $\cur(a, b, \kappa_1, \kappa_2)$, that
was defined in Theorem~\ref{lcur}. The constants $a, b, \kappa_1$ and
$\kappa_2$ will be fixed constants in this section and we shall
therefore refer to $\cur(a, b, \kappa_1, \kappa_2)$ by just $\cur$. 
 For every function $\phi_0 \in \cur$, let us define the {\it local} 
neighborhood $N(\phi_0)$ of $\phi_0$ in $\C$ by  
\begin{equation*}
  N(\phi_0) := \left\{\phi \in \C : \sup_{x \in [0, 1]}|\phi(x) -
    \phi_0(x)| \leq \left(\frac{\kappa_2 c_1^2}{32}\right)^{1/5}
    \left(\frac{\sigma^2}{n} \right)^{2/5} \right\}.  
\end{equation*}
Recall that the constant $c_1$ is defined in~\eqref{eq:DesignPts}. We
define the local minimax risk of $\phi_0 \in \cur$ to be  
\begin{equation*}
  \mix_n(\phi_0) := \inf_{\hat{\phi}} \sup_{\phi \in N(\phi_0)}
  \E_{\phi} \ell^2(\phi, \hat{\phi}),
\end{equation*}
the infimum above being over all possible estimators
$\hat{\phi}$. $\mix_n(\phi_0)$ represents the smallest possible risk
under the knowledge that the unknown convex function $\phi$ lies in
the local neighborhood $N(\phi_0)$ of $\phi_0$. 

In the next theorem, we shall show that the local minimax risk of
every function $\phi_0 \in \cur$ is bounded from below by a constant
multiple of $n^{-4/5}$. Observe that the $l^2$ diameter of $N(\phi_0)$
defined as $\sup_{\phi_1, \phi_2 \in N(\phi_0)} \ell^2(\phi_1,
\phi_2)$ is bounded from above by $n^{-4/5}$ up to multiplicative
factors that are independent of $n$. Therefore, the supremum risk over
$N(\phi_0)$ of any reasonable estimator is bounded from above by
$n^{-4/5}$ up to multiplicative factors. The next theorem shows that if
$\phi_0 \in \cur$, then the supremum risk of every estimator is also
bounded from below by $n^{-4/5}$ up to multiplicative
factors. Therefore, one cannot estimate $\phi_0$ at a
rate faster than $n^{-4/5}$. 

\begin{thm}[Lower bound]\label{ram}
For every $\phi_0 \in \cur(a, b, \kappa_1, \kappa_2)$, we have 
\begin{equation}\label{ram.eq}
  \mix_n(\phi_0) \geq \frac{\kappa_1^2}{4096c_2}
  \left(\frac{\sqrt{c_1}}{\kappa_2}\right)^{8/5}
 (b-a) \left(\frac{\sigma^2}{n} \right)^{4/5}
\end{equation}
provided $n^2 \geq (2c_2)^{5/2} \kappa_2/(\sigma \sqrt{c_1})$. 
\end{thm}

Prototypical examples of functions in $\cur$ include power functions
$x^k$ for $k \geq 2$ and the above theorem implies that every
estimator has rate at least $n^{-4/5}$ for all these functions. Note
that the LSE has the rate $n^{-4/5}$ up to logarithmic factors of $n$ for all functions $\phi_0$. In particular, the LSE is rate optimal (up to logarithmic factors) for all functions in $\cur$. 

Prominent examples of functions not in the class $\cur$ include the
piecewise affine convex functions. As shown in Theorem~\ref{adap},
faster rates are possible for these functions. Essentially, the LSE
converges at the parametric rate (up to logarithmic factors) for these
functions.  

The hardest functions to estimate under the global risk are
therefore smooth convex functions. This is in sharp contrast to the
standpoint of pointwise risk estimation where, for example, cusps in
the function $f(x) = |x|$ are the hardest to estimate. In fact, one
would expect a rate of $n^{-2/3}$ near such cusp points
(see~\citet{CaiLowFwork} for a detailed study of pointwise estimation
although they work with estimators that are different from the
LSE). However, for global estimation, the region over which one gets
such  slower rates is small enough to not effect the overall
near-parametric rate for piecewise affine convex functions. 
  
Our proof of Theorem~\ref{ram} is based on the application of
Assouad's lemma, the following version of which is a consequence of
Lemma 24.3 of~\citet[pp.~347]{vaart98book}. We start by 
introducing some notation. Let $\P_{\phi}$ denote the joint
distribution of the observations $(x_1, Y_1), \dots, (x_n, Y_n)$ when
the true convex function equals $\phi$. For two probability 
measures $P$ and $Q$ having densities $p$ and $q$ with 
respect to a common measure $\mu$, the total variation distance,
$\|P-Q\|_{TV}$, is defined as $\int (|p-q|/2) d\mu$ and the
Kullback-Leibler divergence, $D(P\|Q)$, is defined as $\int p \log
(p/q) d\mu$. Pinsker's inequality asserts
\begin{equation}\label{pins}
  D(P \|Q) \geq 2 \|P - Q \|_{TV}^2 
\end{equation}
for all probability measures $P$ and $Q$. 
\begin{lemma}[Assouad]
Let $m$ be a positive integer and suppose that, for each $\tau \in
\{0, 1\}^m$, there is an associated convex function $\phi_{\tau}$ in
$N(\phi_0)$. Then the following inequality holds:  
  \begin{equation}\label{ass.eq}
    \mix_n(\phi_0) \geq \frac{m}{8} \min_{\tau \neq \tau'}
    \frac{\ell^2(\phi_{\tau}, \phi_{\tau'})}{\ham(\tau, \tau')}
    \min_{\ham (\tau, \tau') = 1} \left(1 - \|\P_{\phi_{\tau}} -
      \P_{\phi_{\tau'}}\|_{TV} \right), 
  \end{equation}
where $\ham(\tau, \tau') := \sum_{i} \{\tau_i \neq \tau'_i\}$. 
\end{lemma}

\begin{proof}[Proof of Theorem~\ref{ram}]
 Fix $m \geq 1$ and consider the same construction $\{\phi_{\tau},
 \tau \in \{0,  1\}^m\}$ from the proof of Theorem~\ref{lcur}. We saw
 there that  
 \begin{equation}\label{ramy}
   \sup_{x \in [0, 1]} |\phi_{\tau}(x) - \phi_0(x)| \leq \frac{(b-a)^2
   \kappa_2}{8 m^2}
 \end{equation}
 and that 
 \begin{equation}\label{borr}
   \ell^2(\phi_{\tau}, \phi_{\tau'}) \geq \ham(\tau, \tau')
   \frac{\kappa_1^2 (b-a)^5}{4096 c_2 m^5}
 \end{equation}
 for every $\tau, \tau' \in \{0, 1\}^m$ provided $n \geq
 4mc_2/(b-a)$. Also, whenever $\ham(\tau, \tau') = 1$, it is clear
 that  
 \begin{equation*}
   \ell^2(\phi_{\tau}, \phi_{\tau'}) \leq \max_{1 \leq i \leq m}
   \ell^2(\phi_0, \max(\phi_0, \alpha_i)). 
 \end{equation*}
 We use Lemma~\ref{buj} to bound $\ell^2(\phi_0, \max(\phi_0,
 \alpha_i))$ from above. Specifically, we use inequality~\eqref{iup}
 with $a = t_{i-1}$ and $b = t_i$ to get 
 \begin{equation*}
   \ell^2(\phi_0, \max(\phi_0, \alpha_i)) \leq \frac{\kappa_2^2 (t_i -
     t_{i-1})^5}{32 c_1} = \frac{\kappa_2^2 (b-a)^5}{32 c_1 m^5}
 \end{equation*}
 provided $n \geq 4mc_1/(b-a)$. Thus under the assumption $n \geq
 4mc_2/(b-a)$,  we have~\eqref{borr} and also (note that $c_2 \geq
 c_1$) 
 \begin{equation*}
   \ell^2(\phi_{\tau}, \phi_{\tau'}) \leq \frac{\kappa_2^2 (b-a)^5}{32
     c_1 m^5} \qt{whenever $\ham(\tau, \tau') = 1$}. 
 \end{equation*}
We apply Assouad's lemma to these functions $\phi_{\tau}$. By
inequality~\eqref{pins}, we get  
\begin{equation*}
  \|\P_{\phi_{\tau}} - \P_{\phi_{\tau'}}\|^2_{TV} \leq \frac{1}{2}
  D(\P_{\phi_{\tau}}\|\P_{\phi_{\tau'}}).
\end{equation*}
By the Gaussian assumption and independence of the errors, the
Kullback-Leibler divergence $D(\P_{\phi_{\tau}}\| P_{\phi_{\tau'}})$
can be easily calculated to be $n \ell^2(\phi_{\tau},
\phi_{\tau'})/(2 \sigma)$. We therefore obtain
\begin{equation*}
  \|\P_{\phi_{\tau}} - \P_{\phi_{\tau'}}\|_{TV} \leq \frac{\sqrt{n}}{2
  \sigma} \ell(\phi_{\tau}, \phi_{\tau'}). 
\end{equation*}
Thus by the application of~\eqref{ass.eq}, we obtain the following 
lower bound for $\mix_n(\phi_0)$: 
\begin{equation}\label{ahos}
  \mix_n(\phi_0) \geq \frac{m}{8} \frac{\kappa_1^2 (b-a)^5}{4096 m^5
    c_2} \left(1 - \frac{\sqrt{n}\kappa_2}{2\sigma}
    \sqrt{\frac{(b-a)^5}{m^5 32 c_1}} \right)
\end{equation}
provided $\phi_{\tau} \in N(\phi_0)$ for each $\tau$. We make the
choice   
\begin{equation*}
  \frac{m}{b-a} := \left(\frac{\sqrt{n}\kappa_2}{\sigma \sqrt{32 c_1}}
  \right)^{2/5}.  
\end{equation*}
The inequality~\eqref{ramy} implies that $\phi_{\tau} \in
N(\phi_0)$. The inequality~\eqref{ram.eq} follows easily
from~\eqref{ahos} . The constraint $n \geq 4c_2m/(b-a)$ 
translates to 
\begin{equation*}
  n^2 \geq (2c_2)^{5/2} \kappa_2/(\sigma \sqrt{c_1}). 
\end{equation*}
The proof is complete. 
\end{proof}

\section{Model misspecification}\label{Misspec}
In this section, we evaluate the performance of the convex LSE
$\hat{\phi}_{ls}$ in the case when the unknown regression function (to
be denoted by $f_0$) is not necessarily convex. Specifically, suppose
that $f_0$ is an unknown function on $[0, 1]$ that is not necessarily
convex. We consider observations $(x_1, Y_1), \dots, (x_n, Y_n)$ from
the model:  
\begin{equation*}
  Y_i = f_0(x_i) + \xi_i, \qt{for $i = 1, \dots, n,$}
\end{equation*}
where $x_1< \dots < x_n$ are fixed design points in $[0, 1]$ and
$\xi_1, \dots, \xi_n$ are independent normal variables with zero mean
and variance $\sigma^2$. 

The convex LSE $\hat{\phi}_{ls}$ is defined in the same way as before as any convex function that minimizes the sum of squares criterion. Since the true function $f_0$ is not necessarily convex, it turns out that the LSE is really estimating the convex projections of $f_0$. Any convex function $\phi_0$ on $[0, 1]$ that
minimizes $\ell^2(f_0, \phi)$ over $\phi \in \C$ is a convex
projection of $f_0$ i.e., 
\begin{equation*}
  \phi_0 \in  \argmin_{\psi \in \C} \sum_{i=1}^n \left(f_0(x_i) -
  \phi(x_i)\right)^2.
\end{equation*}
Convex projections are not unique. However, because $\{(\phi(x_1), 
\dots, \phi(x_n)): \phi \in \C\}$ is a convex closed subset of  
$\R^n$, it follows (see, for example~\citet[Chapter 2]{StarkYang})
that the vector $(\phi_0(x_1), \dots, \phi_0(x_n))$  
is unique for every convex projection $\phi_0$ and, moreover, we have
the inequality:  
\begin{equation}\label{proji}
  \ell^2(f_0, \phi) \geq \ell^2(f_0, \phi_0) + \ell^2(\phi_0, \phi)
  \qt{for every $\phi \in \C$}. 
\end{equation}
The following is the main result of this section. It is the  exact
analogue of Theorem~\ref{twinf} for the case of model
misspecification. 
\begin{thm}\label{mwinf}
  Let $\phi_0$ denote any convex projection of $f_0$ and let $R :=
  \max(1, \lin(\phi_0))$. There exists a positive constant  
  $C$ depending only on the ratio $c_1/c_2$ such that 
  \begin{equation*}
    \E_{f_0} \ell^2(\hat{\phi}_{ls}, \phi_0)  \leq C \left(\log
      \frac{en}{2c_1} \right)^{5/4} \min \left[\left(\frac{\sigma^2
          \sqrt{R}}{n} \right)^{4/5}, \inf_{\alpha \in \Ps}
      \left(\ell^2(\phi_0, \alpha) + \frac{\sigma^2 k^{5/4}(\alpha)}{n}
      \right) \right] 
  \end{equation*}
  provided 
  \begin{equation*}
    n \geq C \frac{\sigma^2}{R^2} \left(\log
        \frac{en}{2c_1} \right)^{5/4}.  
  \end{equation*}
\end{thm}
We omit the proof of this theorem because it is similar to the proof
of Theorem~\ref{twinf}. It is based on the metric entropy results from
Section~\ref{MetricEnt} and the following result from the literature
on the risk behavior of ERMs. 
\begin{thm}\label{genvan}
Let $\phi_0$ denote any convex projection of $f_0$. Suppose $H$ is a
function on $(0, \infty)$ such that   
\begin{equation*}
    H(r) \geq \int_0^r \sqrt{\log M(\epsilon, S(\phi_0, r))} d\epsilon
    \qt{for every $r > 0$}
\end{equation*}
and such that $H(r)/r^2$ is decreasing on $(0, \infty)$. Then there
exists a universal constant $C$ such that  
\begin{equation*}
    \P_{f_0} \left(\ell^2(\hat{\phi}_{ls}, \phi_0) > \delta \right)
    \leq C \sum_{s \geq 0} \exp \left(-\frac{n 2^{2s} \delta}{C^2
        \sigma^2} \right) 
\end{equation*}
for every $\delta > 0$ satisfying $\sqrt{n} \delta \geq C \sigma H(\sqrt{\delta})$. 
\end{thm}
This result is very similar to Theorem~\ref{vandy}. Its proof proceeds
in the same way as the proof of Theorem~\ref{vandy} (see~\citet[Proof
of Theorem 9.1]{VandegeerBook}). We provide below a sketch of its
proof for the convenience of the reader. 
\begin{proof}[Proof of Theorem~\ref{genvan}]
Because $\phi_0$ is convex, we have, by the definition of
$\hat{\phi}_{ls}$, that
\begin{equation*}
\frac{1}{n} \sum_{i=1}^n \left(Y_i - \hat{\phi}_{ls}(x_i) \right)^2
\leq \frac{1}{n}  \sum_{i=1}^n \left(Y_i - \phi_0(x_i) \right)^2.
\end{equation*}
Writing $Y_i = f_0(x_i) + \xi_i$  and simplifying the above
expression, we get
\begin{equation*}
  \ell^2(f_0, \hat{\phi}_{ls}) - \ell^2(f_0, \phi_0) \leq \frac{2}{n}
  \sum_{i=1}^n \xi_i \left(\hat{\phi}_{ls}(x_i) - \phi_0(x_i) \right).
\end{equation*}
Inequality~\eqref{proji} applied with $\phi = \hat{\phi}_{ls}$ gives
\begin{equation*}
  \ell^2(\hat{\phi}_{ls}, \phi_0) \leq \ell^2(f_0, \hat{\phi}_{ls}) -
  \ell^2(f_0, \phi_0).
\end{equation*}
Combining the above two inequalities, we obtain
\begin{equation*}
    \ell^2( \hat{\phi}_{ls}, \phi_0) \leq \frac{2}{n}
  \sum_{i=1}^n \xi_i \left(\hat{\phi}_{ls}(x_i) - \phi_0(x_i)
  \right). 
\end{equation*}
This is of the same form as the  ``basic inequality'' of~\citet[pp.~148]{VandegeerBook}. From here, the proof proceeds just as the proof
of Theorem 9.1 in~\citet{VandegeerBook}. 
\end{proof}
Theorem~\ref{mwinf} shows that one gets
adaptation in the misspecified case provided $f_0$ has a convex
projection that is well-approximable by a piecewise affine convex
function with not too many pieces. An illuminating example of this
occurs when $f_0$ is a  concave function. In this case, we show in
Lemma~\ref{concov} (stated and proved in Section~\ref{AuxRes}) that
$\phi_0$ can be taken to be an affine function, i.e., $\phi_0 \in
\Ps_1$. As a result, it follows that if $f_0$ is concave, then the
risk of $\hat{\phi}_{ls}$ measured from any convex projection of $f_0$
is bounded from above by the parametric rate up to a logarithmic factor
of $n$.  

{\bf Acknowledgements: } The authors would like to thank Aritra Guha, Sasha Tsybakov, a referee and an Associate Editor for their helpful comments.

\appendix
\section{Some auxiliary results}\label{AuxRes}
\begin{lemma}\label{buj}
  Fix $\phi_0 \in \C$ and suppose there exists a subinterval $[a, b]$
  of $[0, 1]$ such that $\phi_0$ is twice differentiable on $[a,
  b]$. Let $\alpha$ denote the linear interpolant of the points $(a,
  \phi_0(a))$ and $(b, \phi_0(b))$ i.e., 
  \begin{equation*}
    \alpha(x) := \phi_0(a) + \frac{\phi_0(b) - \phi_0(a)}{b - a} (x -
    a) \qt{for $x \in [0, 1]$}. 
  \end{equation*}
  \begin{enumerate}
  \item If $\phi_0''(x) \geq \kappa_1$ for all $x \in [a, b]$, then 
    \begin{equation}\label{ilo}
      \ell^2(\phi_0, \max(\phi_0, \alpha)) \geq
      \frac{\kappa_1^2(b-a)^5}{4096 c_2} \qt{when $n \geq
        4c_2/(b-a)$}. 
    \end{equation}
   \item If $\phi_0''(x) \leq \kappa_2$ for all $x \in [a, b]$, then 
     \begin{equation}\label{iup}
       \ell^2(\phi_0, \max(\phi_0, \alpha)) \leq
      \frac{\kappa_2^2(b-a)^5}{32 c_1} \qt{when $n \geq
        4c_1/(b-a)$}. 
     \end{equation}
  \end{enumerate}
\end{lemma}

\begin{proof}[Proof of Lemma~\ref{buj}]
    By convexity of $\phi_0$, it is obvious that $\alpha(x) \geq
    \phi_0(x)$ for $x \in [a, b]$ and $\alpha(x) \leq \phi_0(x)$ for
    $x \notin [a, b]$. We therefore have
    \begin{equation}\label{lacru}
      \ell^2(\phi_0, \max(\phi_0, \alpha)) = \frac{1}{n} \sum_{i=1}^n
      \left(\alpha(x_i) - \phi_0(x_i) \right)^2 I\left\{x_i \in [a, b]
      \right\},
    \end{equation}
    where $I$ denotes the indicator function. By standard error
    estimates for linear interpolation, for every $x \in [a, b]$,
    there exists a point $t_x \in [a, b]$ for which  
    \begin{equation}\label{atki}
      \left|\phi_0(x) - \alpha(x) \right| = (x-a)(b-x)
      \frac{\phi_0''(t_x)}{2}. 
    \end{equation}
    Let us first prove~\eqref{ilo}. By~\eqref{atki} and the assumption
    $\phi_0''(x) \geq \kappa_1$ for $x \in [a, b]$, we have   
    \begin{equation*}
      |\phi_0(x) - \alpha(x)| \geq \frac{(x-a)(b-x) \kappa_1}{2} \qt{for all
        $x \in [a, b]$}. 
    \end{equation*}
    Thus, from~\eqref{lacru}, we get 
   \begin{align*}
      \ell^2(\phi_0, \max(\phi_0,\alpha)) &\geq \frac{\kappa^2_1}{4n}
      \sum_{i=1}^n (x_i - a)^2 (b - x_i)^2 I\left\{x_i \in [a, b]
      \right\} \\
&\geq \frac{\kappa^2_1}{4n} \sum_{i=1}^n (x_i - a)^2 (b - x_i)^2
I\left\{x_i \in [(3a+b)/4, (a+3b)/4]   \right\}. 
   \end{align*}
  Clearly $(x-a)(b-x) \geq (b-a)^2/16$ for every $x \in [(3a+b)/4,
  (a+3b)/4]$ and hence, 
  \begin{equation*}
    \ell^2(\phi_0, \max(\phi_0, \alpha)) \geq \frac{\kappa_1^2}{1024}
    \frac{(b-a)^4}{n} \sum_{i=1}^n I \left\{x_i \in [(3a+b)/4, (a+3b)/4]
    \right\}.
  \end{equation*} 
To get a lower bound on the number of points $x_1, \dots, x_n$ that
are contained in the interval $[(3a+b)/4, (a+3b)/4]$, we use
Lemma~\ref{subm} which gives  
  \begin{equation*}
    \ell^2(\phi_0, \max(\phi_0, \alpha)) \geq \frac{\kappa_1^2}{1024}
    \frac{(b-a)^4}{n} \left(\frac{n(b-a)}{2c_2} - 1 \right). 
  \end{equation*}
The condition $n \geq 4c_2/(b-a)$ now implies that 
  \begin{equation*}
    \frac{n(b-a)}{2c_2} - 1  \geq \frac{n(b-a)}{4c_2} 
  \end{equation*}
which completes the proof of~\eqref{ilo}. We now turn to the proof
of~\eqref{iup}. By~\eqref{atki} and the assumption $\phi_0''(x) \leq
\kappa_2$ for $x \in [a, b]$, we have 
 \begin{equation*}
|\phi_0(x) - \alpha(x)| \leq (x - a)(b-x) \frac{\kappa_2}{2} \qt{for
  all $x \in [a, b]$}. 
 \end{equation*}
Thus from~\eqref{lacru}, we write
  \begin{equation*}
\ell^2(\phi_0, \max(\phi_0 , \alpha)) \leq \frac{\kappa^2_2}{4n}
      \sum_{i=1}^n (x_i - a)^2 (b - x_i)^2 I\left\{x_i \in [a, b]
      \right\} .
  \end{equation*}
  Because $(x-a)(b-x) \leq (b-a)^2/4$ for all $x \in [a, b]$, we
  obtain 
  \begin{equation*}
    \ell^2(\phi_0, \max(\phi_0, \alpha)) \leq \frac{\kappa_2^2}{64}
    \frac{(b-a)^4}{n} \sum_{i=1}^n I\left\{x_i \in [a, b] \right\}. 
  \end{equation*}
To obtain an upper bound on the number of points $x_1, \dots, x_n$
that are contained in $[a, b]$, we again use Lemma~\ref{subm} to get 
  \begin{equation*}
    \ell^2(\phi_0, \max(\phi_0, \alpha)) \leq \frac{\kappa_2^2}{64}
    \frac{(b-a)^4}{n} \left(\frac{n(b-a)}{c_1} + 1 \right)
  \end{equation*}
When $n \geq 4c_1/(b-a)$, we have 
  \begin{equation*}
    \frac{n(b-a)}{c_1} + 1 \leq \frac{2n(b-a)}{c_1}
  \end{equation*}
  and this completes the proof. 
  \end{proof}

\begin{lemma}\label{subm}
  Let $x_1 < \dots < x_n$ be fixed points in $[0, 1]$ satisfying $c_1
  \leq n(x_i - x_{i-1}) \leq c_2$ for all $2 \leq i \leq n$. Let $[a,
  b]$ be a subinterval of $[0, 1]$ that contains $m$ of the $n$ real
  numbers $x_1, \dots, x_n$. Then 
  \begin{equation}\label{subm.eq}
\frac{n(b-a)}{c_2} - 1 \leq   m \leq \frac{n(b-a)}{c_1} + 1. 
  \end{equation}
\end{lemma}
\begin{proof}
Let $x_0 := \max \left( x_1 - c_2/n, 0 \right)$ and $x_{n+1} :=
  \min \left(x_n + c_2/n, 1 \right)$. Let  
  \begin{equation*}
    \left\{x_1, \dots, x_n \right\} \cap [a, b] = \left\{x_{k+1},
      \dots, x_{k+m} \right\}
  \end{equation*}
  for some $0 \leq k \leq n-m$. Clearly
  \begin{equation*}
    b - a \geq x_{k+m} - x_{k+1} = \sum_{i=k+2}^{k+m} \left(x_i -
      x_{i-1} \right) \geq \frac{c_1 (m-1)}{n}
  \end{equation*}
  which gives the upper bound in~\eqref{subm.eq}. On the other hand, 
  \begin{equation*}
    b-a \leq x_{k+m+1} - x_k = \sum_{i=k+1}^{k+m+1} \left(x_{i} -
      x_{i-1} \right) \leq \frac{c_2(m+1)}{n}
  \end{equation*}
  which gives the lower bound in~\eqref{subm.eq}. The proof is
  complete. 
\end{proof}

\begin{lemma}\label{1db}
  Let $\phi$ be a convex function on $[0, 1]$ for which $\int_0^1
  |\phi(x)|^p dx \leq 1$ for a fixed $p \geq 1$. Then $|\phi(y)| \leq
  2(1+p)^{1/p} \max \left(y^{-1/p}, (1-y)^{-1/p} \right)$ for all $y
  \in (0, 1)$. 
\end{lemma}
\begin{proof}
It suffices to prove the theorem for $0 < y < 1/2$. 

Suppose $\phi(y) > y^{-1/p}$. Then, by convexity of $\phi$, the
condition $\phi(x) > \phi(y)$ must hold either for all $x \in (0, y)$
or for all $x \in (y, 1)$. Therefore,  
\begin{equation*}
  1 \geq \int |\phi(x)|^p dx \geq \phi(y)^p \min(y, 1-y) \geq
  \phi(y)^p y
\end{equation*}
which gives a contradiction. Therefore $\phi(y) \leq y^{-1/p}$. 

Suppose, if possible, that $\phi(y) < -c y^{-1/p}$ for some $c >
1$. We consider the following cases separately.

Case ($i$): Assume $\phi(0) < -cy^{-1/p}$ . In this case,
  by convexity of $\phi$, it follows that $\phi(x) < -cy^{-1/p}$ for
  all $x \in [0, y]$. Therefore $|\phi(x)| > cy^{-1/p}$ and thus
  \begin{equation*}
   1 \geq \int_0^1 |\phi(x)|^p dx \geq \int_0^y \frac{c^p}{y} dx =
   c^p. 
  \end{equation*}
  This contradicts $c > 1$. 

Case ($ii$): Here $\phi(0) \geq -cy^{-1/p}$. We now consider the following
  two subcases: 
  \begin{enumerate}
  \item $\phi(0) \leq 0$. Then $\phi(x) \leq 0$ for all $x \in [0,
    y]$. For each $0 \leq x \leq y$, we have, by convexity, 
    \begin{equation*}
      \phi(x) \leq \left(1-\frac{x}{y} \right) \phi(0) + \frac{x}{y}
      \phi(y) \leq \frac{x}{y} \phi(y). 
    \end{equation*}
  Thus $y \phi(x) \leq x \phi(y) \leq 0$ for each $0 \leq x \leq
  y$. As a result, 
  \begin{equation*}
    y^p |\phi(x)|^p \geq x^p |\phi(y)|^p \qt{for $0 \leq x \leq y$}. 
  \end{equation*}
Integrating both sides from $x = 0$ to $x = y$, we obtain
  \begin{equation*}
    y^p \int_0^y |\phi(x)|^p dx \geq |\phi(y)|^p \frac{y^{p+1}}{p+1}
  \end{equation*}
  which implies that $|\phi(y)|^{p} \leq (p+1)/y$, i.e., $|\phi(y)| \leq
  (1+p)^{1/p} y^{-1/p}$ which is a contradiction if $c >
  (1+p)^{1/p}$. 
\item $\phi(0) > 0$. Let $z \in (0, y)$ be such that $\phi(z) =
  0$. For $x < z$, we can write, by convexity, 
  \begin{equation*}
    0 = \phi(z) \leq \frac{y-z}{y-x}\phi(x) + \frac{z-x}{y-x}
    \phi(y)
  \end{equation*}
which implies that 
\begin{equation*}
0 > \phi(y) \geq \frac{y-z}{x-z} \phi(x).
\end{equation*}
As a result, $|z-x|^{p} |\phi(y)|^p \leq |y-z|^p |\phi(x)|^p$ for $0 <
x < z$. Integrating both sides from $x = 0$ to $x = z$, we get 
\begin{equation}\label{a1}
  |\phi(y)|^p \frac{z^{p+1}}{p+1} \leq |y-z|^p \int_0^z |\phi(x)|^p
  dx. 
\end{equation}
For $z < x < y$, again, by convexity, we write
\begin{equation*}
  \phi(x) \leq \frac{x-z}{y-z} \phi(y) + \frac{y-x}{y-z} \phi(z) =
  \frac{x-z}{y-z} \phi(y) \leq 0. 
\end{equation*}
As a result, $|y-z|^p |\phi(x)|^p \geq |x-z|^p
|\phi(y)|^p$. Integrating from $x = z$ to $x = y$, we get
\begin{equation}\label{a2}
  |\phi(y)|^p \frac{(y-z)^{p+1}}{p+1} \leq |y-z|^p \int_z^y
  |\phi(x)|^p dx. 
\end{equation}
Adding the two inequalities~\eqref{a1} and~\eqref{a2}, we obtain
\begin{equation*}
 \frac{|\phi(y)|^p}{p+1} \left(z^{p+1} + (y-z)^{p+1} \right) \leq
 |y-z|^p \int_0^y |\phi(x)|^p dx <  y^p . 
\end{equation*}
Now 
\begin{equation*}
  z^{p+1} + (y-z)^{p+1} \geq \min_{0 < u < y}  \left(u^{p+1} +
    (y-u)^{p+1} \right) = 2^{-p} y^{p+1}. 
\end{equation*}
Combining, we obtain
\begin{equation*}
  |\phi(y)| < 2 (1+p)^{1/p} y^{-1/p}
\end{equation*}
which results in a contradiction if $c \geq 2(1+p)^{1/p} y^{-1/p}$. 
  \end{enumerate}
\end{proof}

\begin{lemma}[Interpolation Lemma]\label{inter}
Fix $x_1 < x_2 < \dots < x_n$ and suppose that
$c_1 \leq n(x_i - x_{i-1}) \leq c_2$ for all $2 \leq i \leq n$. For
every function $f$ on $[x_1, x_n]$, associate another function 
$\tilde{f}$ on $[x_1, x_n]$ by  
\begin{equation*}
  \tilde{f}(x) := \frac{x_{i+1} - x}{x_{i+1} - x_i}f(x_i)  +
  \frac{x-x_i}{x_{i+1} - x_i} f(x_{i+1}) \qt{for $x_i \leq x \leq x_{i+1}$}
\end{equation*}
where $i = 1, \dots, n-1$. Then for every pair of functions $f$ and
$g$ on $[x_1, x_n]$, we have  
\begin{equation*}
\frac{1}{c_2} \int_{x_1}^{x_n} \left(\tilde{f}(x) - \tilde{g}(x)
\right)^2 dx \leq \frac{1}{n} 
\sum_{i=1}^n \left(f(x_i) - g(x_i) \right)^2  \leq \frac{6}{c_1}
\int_{x_1}^{x_n} \left(\tilde{f}(x) - \tilde{g}(x) 
\right)^2 
dx. 
\end{equation*}  
\end{lemma}
\begin{proof}
 It is elementary to check that for every $1 \leq i \leq n-1$, we have
\begin{equation*}
  \int_{x_i}^{x_{i+1}} \left(\tilde{f}(x) - \tilde{g}(x) \right)^2 dx
  = \frac{x_{i+1} - x_i}{3} \left(\alpha^2 + \beta^2 + \alpha \beta \right) 
\end{equation*}
where $\alpha := f(x_i) - g(x_i)$ and $\beta = f(x_{i+1}) -
g(x_{i+1})$. Using the inequalities 
\begin{equation*}
  \frac{-\alpha^2 - \beta^2}{2} \leq \alpha \beta \leq \frac{\alpha^2
    + \beta^2}{2},
\end{equation*}
we obtain 
\begin{eqnarray*}
\frac{c_1(\alpha^2 + \beta^2)}{6n} \leq \left(x_{i+1} - x_i \right)
\frac{\alpha^2 + \beta^2}{6} & \leq & \int_{x_i}^{x_{i+1}} \left(\tilde{f}(x) - \tilde{g}(x) \right)^2 dx \\
& \leq & \left(x_{i+1} - x_i \right) \frac{\alpha^2 + \beta^2}{2} \leq
\frac{c_2(\alpha^2 + \beta^2)}{2n}. 
\end{eqnarray*}
Adding these inequalities from $i = 1$ to $i = n-1$, we deduce 
\begin{equation*}
\frac{c_1}{6n} \sum_{i=1}^n \left(f(x_i) - g(x_i) \right)^2 \leq 
\int_{x_1}^{x_n} \left(\tilde{f}(x) - \tilde{g}(x) \right)^2 dx \leq 
  \frac{c_2}{n} \sum_{i=1}^n \left(f(x_i) - g(x_i) \right)^2
\end{equation*}
which yields the desired result. 
\end{proof}
\begin{remark}
Observe that if $f$ is a convex function on $[a, b]$, then $\tilde{f}$
is also convex on $[a, b]$. 
\end{remark}

\begin{lemma}\label{concov}
The set of all convex projections of a concave function $f_0$ includes an affine function.  
\end{lemma}
\begin{proof}
We prove this result by the method of contradiction. Suppose that there is no convex projection that is affine. Let $\phi_0$ be the continuous piecewise affine convex projection of $f_0$. For a function $g:[0,1] \rightarrow \R$ we define $g(0+) := \lim_{x \rightarrow 0 +} g(x)$ and $g(1-) := \lim_{x \rightarrow 1 -} g(x)$. This notation is necessary as $f_0$ need not be continuous at the boundary points $\{0,1\}$. 
	
Case ($i$): Suppose that $f_0(0+) \ge \phi_0(0)$ and $f_0(1-) \ge \phi_0(1)$. Then the affine function $\tilde {\phi}_0$ obtained by joining $(0,\phi_0(0))$ and $(1,\phi_0(1))$, i.e., $\tilde {\phi}_0(x) = (1-x)\phi_0(0) + x \phi_0(1)$, for $x \in [0,1]$, lies in-between $\phi_0$ and $f_0$ (as $f_0$ is concave) and $\ell^2(\phi_0,f_0) \ge \ell^2(\tilde{\phi}_0,f_0)$, giving rise to a contradiction. 
	
	Case ($ii$): Suppose that $f_0(0+) < \phi_0(0)$ and $f_0(1-) \ge \phi_0(1)$. Then there is a point $u \in (0,1)$ such that $f_0(u) = \phi_0(u)$. Let us define $\tilde \phi$ to be the affine function joining $(u,\phi_0(u))$ and $(1,\phi_0(1))$. Again, $\tilde{\phi}_0$ lies in-between $\phi_0$ and $f_0$ and $\ell^2(\phi_0,f_0) \ge \ell^2(\tilde{\phi}_0,f_0)$, thus giving rise to a contradiction. 
	
	Case ($iii$): Suppose that $f_0(0+) \ge \phi_0(0)$ and $f_0(1-) < \phi_0(1)$. A similar analysis as in ($ii$) by looking at the affine function obtained by joining $(0,\phi_0(0))$ and $(v,\phi_0(v))$ where $\phi_0(v) = f_0(v)$, $v \in (0,1)$, gives a contradiction.
	
		Case ($iv$): Suppose that $f_0(0+) < \phi_0(0)$ and $f_0(1-) <\phi_0(1)$. Suppose that there are two points $u_0, u_1 \in (0,1)$ such that $f_0(u_i) = \phi_0(u_i)$, for $i = 1,2$. Then define $\tilde \phi$ to be the affine function joining $(u_0, \phi_0(u_0))$ and $(u_1, \phi_0(u_1))$. Again, $\tilde{\phi}_0$ lies in-between $\phi_0$ and $f_0$ and $\ell^2(\phi_0,f_0) \ge \ell^2(\tilde{\phi}_0,f_0)$, thus giving rise to a contradiction. Suppose that $f_0$ and $\phi_0$ touch at just one point $v \in (0,1)$. Then defining $\tilde{\phi}_0$ to be the affine function that passes through $(v,\phi_0(v))$ and is a sub-gradient to both $\phi_0$ and $f_0$ at $v$ yields a contradiction. If $f_0$ and $\phi_0$ do not touch at all then defining $\tilde{\phi}_0$ to be any affine function lying between $\phi_0$ and $f_0$ shows that $\ell^2(\phi_0,f_0) \ge \ell^2(\tilde{\phi}_0,f_0)$.  This completes the proof.
\end{proof}
\begin{remark}
Note that if $n > 2$, the convex projection of a concave $f_0$ is in fact unique on $(0,1)$ and affine.
\end{remark}


%

\bibliographystyle{apalike}
\bibliography{AG}

\def\noopsort#1{}
\begin{thebibliography}{}

\bibitem[Atkinson, 1989]{Atkinson88}
Atkinson, K.~E. (1989).
\newblock {\em An introduction to numerical analysis}.
\newblock John Wiley \& Sons Inc., New York, second edition.

\bibitem[Barron et~al., 1999]{BarronBirgeMassart}
Barron, A., Birg{\'e}, L., and Massart, P. (1999).
\newblock Risk bounds for model selection via penalisation.
\newblock {\em Probability Theory and Related Fields}, 113:301--413.

\bibitem[Birg{\'e}, 1989]{Birge89}
Birg{\'e}, L. (1989).
\newblock The {G}renander estimator: a nonasymptotic approach.
\newblock {\em Ann. Statist.}, 17(4):1532--1549.

\bibitem[Birg\'e and Massart, 1993]{BM93}
Birg\'e, L. and Massart, P. (1993).
\newblock Rates of convergence for minimum contrast estimators.
\newblock {\em Probability Theory and Related Fields}, 97:113--150.

\bibitem[Bronshtein, 1976]{Bronshtein76}
Bronshtein, E.~M. (1976).
\newblock $\epsilon$-entropy of convex sets and functions.
\newblock {\em Siberian Mathematical Journal}, 17:393--398.

\bibitem[Cai and Low, 2014]{CaiLowFwork}
Cai, T. and Low, M. (2014).
\newblock A framework for estimation of convex functions.
\newblock {\em Statist. Sinica (to appear)}.
\newblock Available at
  http://www3.stat.sinica.edu.tw/ss\_newpaper/SS-13-279\_na.pdf.

\bibitem[Carolan and Dykstra, 1999]{CD99}
Carolan, C. and Dykstra, R. (1999).
\newblock Asymptotic behavior of the {G}renander estimator at density flat
  regions.
\newblock {\em Canad. J. Statist.}, 27(3):557--566.

\bibitem[Cator, 2011]{Cator11}
Cator, E. (2011).
\newblock Adaptivity and optimality of the monotone least-squares estimator.
\newblock {\em Bernoulli}, 17(2):714--735.

\bibitem[Chatterjee et~al., 2013]{CGS13}
Chatterjee, S., Guntuboyina, A., and Sen, B. (2013).
\newblock Improved risk bounds in isotonic regression.
\newblock available at http://http://arxiv.org/abs/1311.3765.

\bibitem[Cule et~al., 2010]{CSS10}
Cule, M.~L., Samworth, R.~J., and Stewart, M.~I. (2010).
\newblock Maximum likelihood estimation of a multi-dimensional log-concave
  density (with discussion).
\newblock {\em Journal of the Royal Statistical Society, Series B},
  72:545--600.

\bibitem[Dryanov, 2009]{Dryanov}
Dryanov, D. (2009).
\newblock Kolmogorov entropy for classes of convex functions.
\newblock {\em Constructive Approximation}, 30:137--153.

\bibitem[D{\"u}mbgen et~al., 2004]{DuembgenEtAl04}
D{\"u}mbgen, L., Freitag, S., and Jongbloed, G. (2004).
\newblock Consistency of concave regression with an application to
  current-status data.
\newblock {\em Math. Methods Statist.}, 13(1):69--81.

\bibitem[Dykstra, 1983]{Dykstra83}
Dykstra, R.~L. (1983).
\newblock An algorithm for restricted least squares regression.
\newblock {\em J. Amer. Statist. Assoc.}, 78(384):837--842.

\bibitem[Fraser and Massam, 1989]{FraserM89}
Fraser, D. A.~S. and Massam, H. (1989).
\newblock A mixed primal-dual bases algorithm for regression under inequality
  constraints. {A}pplication to concave regression.
\newblock {\em Scand. J. Statist.}, 16(1):65--74.

\bibitem[Grenander, 1956]{G56}
Grenander, U. (1956).
\newblock On the theory of mortality measurement. {II}.
\newblock {\em Skand. Aktuarietidskr.}, 39:125--153 (1957).

\bibitem[Groeneboom et~al., 2001]{GroeneboomEtAl01}
Groeneboom, P., Jongbloed, G., and Wellner, J.~A. (2001).
\newblock Estimation of a convex function: characterizations and asymptotic
  theory.
\newblock {\em Ann. Statist.}, 29(6):1653--1698.

\bibitem[Groeneboom and Pyke, 1983]{GP83}
Groeneboom, P. and Pyke, R. (1983).
\newblock Asymptotic normality of statistics based on the convex minorants of
  empirical distribution functions.
\newblock {\em Ann. Probab.}, 11(2):328--345.

\bibitem[Guntuboyina and Sen, 2013]{GS13}
Guntuboyina, A. and Sen, B. (2013).
\newblock Covering numbers for convex functions.
\newblock {\em IEEE Trans. Inf. Th.}, 59(4):1957--1965.

\bibitem[Hanson and Pledger, 1976]{HanPled76}
Hanson, D.~L. and Pledger, G. (1976).
\newblock Consistency in concave regression.
\newblock {\em Ann. Statist.}, 4(6):1038--1050.

\bibitem[Hildreth, 1954]{Hildreth54}
Hildreth, C. (1954).
\newblock Point estimates of ordinates of concave functions.
\newblock {\em J. Amer. Statist. Assoc.}, 49:598--619.

\bibitem[Jankowski, 2014]{HW12}
Jankowski, H. (2014).
\newblock Convergence of linear functionals of the {G}renander estimator under
  misspecification.
\newblock {\em Ann. Statist.}, 42(2):625--653.

\bibitem[Mammen, 1991]{Mammen91}
Mammen, E. (1991).
\newblock Nonparametric regression under qualitative smoothness assumptions.
\newblock {\em Ann. Statist.}, 19(2):741--759.

\bibitem[Massart, 2007]{Massart03Flour}
Massart, P. (2007).
\newblock {\em Concentration inequalities and model selection. Lecture notes in
  Mathematics}, volume 1896.
\newblock Springer, Berlin.

\bibitem[Pollard, 1990]{Pollard90Iowa}
Pollard, D. (1990).
\newblock {\em Empirical Processes: Theory and Applications}, volume~2 of {\em
  NSF-CBMS Regional Conference Series in Probability and Statistics}.
\newblock Institute of Mathematical Statistics, Hayward, CA.

\bibitem[Rigollet and Tsybakov, 2012]{RT12}
Rigollet, P. and Tsybakov, A.~B. (2012).
\newblock Sparse estimation by exponential weighting.
\newblock {\em Statist. Sci.}, 27(4):558--575.

\bibitem[Seijo and Sen, 2011]{SS11}
Seijo, E. and Sen, B. (2011).
\newblock Nonparametric least squares estimation of a multivariate convex
  regression function.
\newblock {\em Annals of Statistics}, 39:1633--1657.

\bibitem[Seregin and Wellner, 2010]{SW10}
Seregin, A. and Wellner, J.~A. (2010).
\newblock Nonparametric estimation of multivariate convex-transformed
  densities.
\newblock {\em Annals of Statistics}, 38:3751--3781.

\bibitem[Stark and Yang, 1988]{StarkYang}
Stark, H. and Yang, Y. (1988).
\newblock {\em Vector space projections}.
\newblock Wiley, New York.

\bibitem[Van~de Geer, 1993]{vdG93}
Van~de Geer, S. (1993).
\newblock Hellinger-consistency of certain nonparametric maximum likelihood
  estimators.
\newblock {\em Ann. Statist.}, 21(1):14--44.

\bibitem[Van~de Geer, 2000]{VandegeerBook}
Van~de Geer, S. (2000).
\newblock {\em Applications of Empirical Process Theory}.
\newblock Cambridge University Press.

\bibitem[Van~der Vaart, 1998]{vaart98book}
Van~der Vaart, A. (1998).
\newblock {\em Asymptotic Statistics}.
\newblock Cambridge University Press.

\bibitem[van~der Vaart and Wellner, 1996]{vaartwellner96book}
van~der Vaart, A.~W. and Wellner, J.~A. (1996).
\newblock {\em Weak Convergence and Empirical Process: With Applications to
  Statistics}.
\newblock Springer-Verlag.

\bibitem[Zhang, 2002]{Zhang02}
Zhang, C.-H. (2002).
\newblock Risk bounds in isotonic regression.
\newblock {\em Ann. Statist.}, 30(2):528--555.

\end{thebibliography}
\end{document}